\documentclass[11pt,a4paper]{article}

\usepackage{dsfont}

\usepackage{amssymb}
\usepackage{amsmath}
\usepackage{amsthm}
\usepackage{amsfonts}
\usepackage{epsfig}
\usepackage{marvosym}

\usepackage[english]{babel}
\usepackage{color}

\newtheorem{definition}{Definition}
\newtheorem{proposition}[definition]{Proposition}
\newtheorem{corollary}[definition]{Corollary}
\newtheorem{lemma}[definition]{Lemma}
\newtheorem{theorem}[definition]{Theorem}
\newtheorem{remark}[definition]{Remark}

\newcommand{\R}{\mathbb{R}}

\usepackage{setspace}
\newsavebox{\fmbox}

\usepackage{multirow}
\usepackage{graphicx}

\begin{document}

\title{Modelling the inflammatory process in atherosclerosis: a nonlinear renewal equation}

\author{Nicolas Meunier \and Nicolas Muller \thanks{MAP5, CNRS UMR 8145, Universit\'{e} Paris Descartes, 45 rue des Saints  P\`{e}res 75006 Paris, France. ({\tt nicolas.meunier@parisdescartes.fr, nicolas.muller@parisdescartes.fr})}. }

\maketitle

\begin{abstract}
We present here a population structured model to describe the dynamics of macrophage cells. The model involves the interactions between modified LDL, monocytes/macrophages, cytokines and foam cells. The key assumption is that the individual macrophage dynamics depends on the amount of lipoproteins it has internalized. The obtained renewal equation is coupled with an ODE describing the lipoprotein dynamics.  We first prove global existence and uniqueness for the nonlinear and nonlocal system. We then study long time asymptotics in a particular case describing silent plaques which undergo periodic rupture and repair. Finally we study long time asymptotics for the nonlinear renewal equation obtained when considering the steady state of the ODE. and we prove that .... 
%This model suggests that there is an initial inflammatory phase associated with atherosclerotic lesion development and a longer, quasi-static process of plaque development inside the arterial wall that follows the initial transient. We will show results that show how different low density lipoprotein (LDL) concentrations in the blood stream and different immune responses can affect the development of a plaque. Through numerical bifurcation analysis, we show the existence of a fold bifurcation when the flux of LDL from the blood is sufficiently high. By analysing the model presented in this paper, we gain a greater insight into this inflammatory response qualitatively and quantitatively.
\end{abstract}

%\begin{AMS}
%35B60; 35B44; 35Q92; 92C17; 92B05.
%\end{AMS}

%\tableofcontents

\section{Introduction}
\noindent
 
Atherosclerosis is the leading cause of death in industrialized societies since it is the primary cause of heart attack (acute myocardial infarction) and stroke (cerebrovascular accident). It is now accepted that atherosclerosis is a chronic inflammatory disease which starts within the intima, the innermost layer of an artery. It is driven by the accumulation of macrophage cells within the intima and promoted by modified low density lipoprotein (LDL) particles \cite{Libby, Lusis, Tabas}.

The early stages of atherosclerosis are non-symptomatic. Atherosclerotic lesions, or plaques, can form in the artery wall as early as young  and continue to evolve and grow throughout adult life. During this evolution and growth, symptoms may begin to occur. Indeed, plaques stiffen the arterial wall, and when covered by a thin fibrous cap they may become unstable and rupture, causing occlusions to blood flow to vital organs which may result in heart attack and stroke.

Plaque formation, growth and rupture depend on many nonlinear biophysical processes that occur on different timescales and in different locations. In particular, some processes take place in lining of the blood vessel and others in the vessel wall. 

Over the past decades, the vulnerable plaque, defined as a lesion loaded with lipid, macrophage rich, covered by a thin fibrous cap, has been considered poised to rupture. Consequently the thin-capped fibroatheroma lesion has become a target for imaging, possible intervention, model attempts in animals, and much discussion. The concept of vulnerable plaque has proven highly useful to guide research and thinking regarding the plaque risk factor.

In this work we propose a structured population model to describe the dynamics of macrophages in the intima. Our key assumption is that macrophage inflammatory activity depends on the amount of internalized LDL.

%The mathematical model we present here is based on the assumption that immune cell (macrophage) inflammatory state is characterized by its "size" $a$, a one-dimensional quantity which amounts for the LDL the cell has internalized. 

%More precisely,  we consider the following nonlinear first order partial differential equation to describe the macrophage density dynamics: 
%\begin{equation}\label{eq:model_intro}
%\partial_t M(t, a) + C (t)\, \partial_ a\big(\mathcal{V}(a) \, M(t, a)\big) + \mu (a) \, M(t, a) = 0\, , \quad t>0\, , a>0\, ,
%\end{equation}
%completed by a nonlinear and nonlocal boundary term:
%\begin{equation}\label{eq:bord_intro}
%M(t,0) =f  \left(C(t) ,\int_{0}^{+ \infty } B(a) \, M(t,a) \, da \right),
%\end{equation}
%where $C(t)$, the concentration of modified LDL in the intima, satisfies an ordinary differential equation: \begin{equation}\label{eq:Cprime_intro}
%C'(t) = R(C(t)) - C(t)\, \int_{0}^{+ \infty } \mathcal{V}(a) \, M(t,a) \, da.
%\end{equation}

Related models featuring structured population balances have been widely used to model biological population dynamics, for instance,  stability of structured populations was investigated in \cite{Farkas2007119, Farkas2010}, existence of solutions of continuous diffusive coagulation– fragmentation models was studied in \cite{Amann2005159} e.g., similar models without diffusion were discussed in \cite{Giri2012,Rudnicki2010,Morale2004}, existence of solutions using a fixed-point approach  was investigated in \cite{Farkas2012,MichelPerthame, Perthame_transport}, and finally size distribution and long-time asymptotics were studied in \cite{Sizedistribution,Gabriel2012}.

Using a fixed-point approach we prove existence and uniqueness of the solution to system \eqref{Pbcouple} defined below. This system consists in a nonlinear and nonlocal renewal equation coupled to an ordinary differential equation.

%describing a "young" or a slow evolutive lesion we obtain a Lokta-Volterra system  which exibits periodic solution. In the context of atherosclerosis, such a periodic behaviour is in agreement with biological observations, namely "young" or a slow evolutive lesions undergo erosion and repair    where the LDL infiltration we also indicate how more general situations (size-dependent growth, temperature dependence of saturation constant, and agglomeration) can be integrated in the setting. Our proof expands on Gurtin and MacCamy [19] and Calsina [20] (see also [21, 22]). The full model for which existence of solutions is proved comprises Equations (11)–(16), which we derive in the following sections.

The structure of the paper is as follows. In Sections 2, the model is derived from population and mass balances. Existence and uniqueness are proved in Section 3. Long time asymptotics is discussed in Section 4. Finally, a simplified high inflammatory model is presented in Section 5.

\begin{figure}
\begin{center}
\includegraphics[width=.4\linewidth]{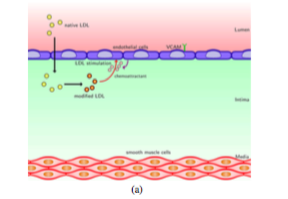}
\includegraphics[width=.37\linewidth]{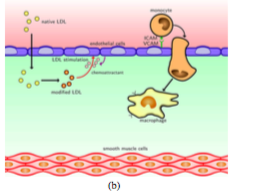}
\includegraphics[width=.4\linewidth]{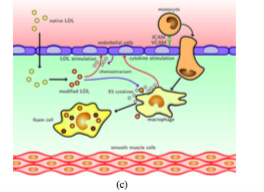}
 \end{center} 
\caption{(a) When endothelium is injuried, LDL from the lumen infiltrates the intima. LDL inside the intima is susceptible to modification. The presence of modified LDL at the endothelium triggers the production of chemoattractant and stimulating factors. This aggravation of the endothelium triggers an expression on the surface of the endothelial wall of adhesion molecules such as VCAM-1 and ICAM-1. (b) Monocytes from the blood stream adhere to the endothelial wall and move into the intima in response to the chemoattractant present in the intima. Once inside the intima, macrophage colony stimulating factor (M-CSF) stimulates the monocytes to differentiate into macrophages and these macrophages display scavenger receptors on their surface. (c) Scavenger expression allows macrophages to consume the modified LDL inside the intima. Macrophages secrete chemoattractants such as MCP-1 and endothelial stimulating cytokines (e.g. TNFα). As they consume modified LDL, macrophages become foam cells and cease to consume modified LDL}\label{Leoni}
\end{figure}

%Our purpose in this paper is to study several aspects of this model: global existence of solution and under appropriate assumptions description of the long-time profile.  We also perform numerical simulations to illustrate the theoritical approach. 

%The plan of this article is the following. In the second section, we detail the construction of the model. \textbf{A refaire} In the third section, we prove the convergence towards the steady state and in the last section, we study the dependence between size of the necrotic core and the apoptotic rate. 

\section{Model description}
\noindent

In a healthy artery the endothelium, which is a layer of cells, covers the inside of the vessel wall and contributes to the integrity of the arterial wall. In areas of low shear stress, the endothelium can become damaged. Once the endothelium has been injured, LDL, which is present in the blood stream, penetrates the arterial wall and enters the intima. Inside the intima free radicals, which are produced when the endothelial wall fails, may modify LDL. The failure of the endothelium and the presence of modified LDL trigger the expression of adhesion molecules on the surface of the endothelium enabling monocytes from the blood stream to attach themselves to the endothelial wall. Then monocytes migrate into the intima in response to a chemoattractive gradient established by cytokines, \cite{Grec}. 
%The phenomenon of monocyte transmigration from the blood flow into the intima is called monocyte extravasation. 
Once in the arterial wall, monocytes differentiate into macrophages.  Macrophages that exhibit scavenger receptors recognise modified LDL in the intima and internalise the lipid content. Stimulated by their consumption of modified LDL, macrophages secrete chemoattractants and endothelial-stimulating cytokines  that both promote further immigration of macrophages, Fig 1. Lipid-laden macrophages become foam cells, that is, large cells which are filled with lipid droplets that give them a foamy appearance. 
%Foam cells collectively form a fatty streak inside the intima.

Over weeks, months, and even years, the previously described process continues and may even amplify. Foam cells secrete additional molecules that further promote lipoprotein retention and extravasation of more monocytes. After apoptosis, foam cells form the lipid-rich necrotic core, and contribute to the volume growth of the intima eventually restricting the artery lumen. Some of these lesions, fed by the failure of inflammation resolution, undergo a process of arterial wall breakdown. These latter types of lesions are referred to as vulnerable plaques. Plaque erosion and then plaque rupture can expose pro-coagulant and pro-thrombotic contents in the intima to coagulation factors and platelets in the bloodstream. In the worse case scenario, an occlusive thrombus forms, leading to acute oxygen and nutrient deprivation of distal tissues fed by the artery. When these events occur in the coronary arteries the result is unstable angina, myocardial infarction, or sudden cardiac death and when it occurs in the carotid arteries the result is cerebral stroke.

A major goal is to predict plaque evolution. There are currently many models for early atherosclerosis description, see e.g. \cite{Genieys07,Genieys09,Genieys12,CEMRACS,CEMRACS2,Chalmers2015}, but no comprehensive model for vulnerable plaque description. Models assume homogeneity of cells within population classes. However, macrophage cells differ in their physiological and inflammatory characteristics. Furthermore it is now accepted that a key event in atherosclerosis is a deleterious inflammatory response to subendothelial lipoproteins \cite{Tabas}. A critical aspect of this response is failure of inflammation resolution, which normally consists in inflammatory cell influx regulation together with effective clearance of apoptotic cells and promotion of inflammatory cell exit. Defects in one of these processes promote the progression of atherosclerotic lesions into vulnerable plaques. Since monocyte influx, evolution into first macrophage and then foam cell, and foam cell apoptosis depend on LDL concentration in the intima, it is natural to use a structured population modelling to account for these differences. Furthermore, the knowledge of both inflammatory activity and lipid content is necessary to predict clinical events. For example, in the case of low inflammatory plaques with low lipid content, rupture is a common phenomenon. But in this case the plaque rupture is silent as it is immediately followed by repair and it does not cause any symptom, \cite{Virmani2004}. 
%It is useful to describe LDL concentration dynamics.

Here, we group and structure the
% monocyte -- macrophage -- foam cell dynamics according to the amount of lipids the cell has consumed. This leads to a renewal one-dimensional PDE.  The 
macrophage population, denoted by $M(t,a)$, into a single species by the amount of internalized lipoproteins $a\ge 0$: $M(t,0)$ represents the incoming monocyte concentration, differentiation into macrophages occurs for $a\in (0, A)$, evolution into foam cells for $a\in (A, A_1)$ and apoptosis if $a>A_1$.

Macrophages internalize lipoproteins, this ability increases with the density of lipoproteins already present in the intima, $C(t)$, and it also depends on the quantity of LDL a macrophage has already internalized, $a$. We assume the speed of internalization to be given by $C(t) \mathcal{V}(a)$ with $\mathcal{V}$ a positive function. Furthermore, individuals of the macrophage population die with a rate $\mu(a)$ depending only on the amount of internalized LDL. The dynamics of this population is given by the nonlinear and nonlocal transport equation:
\begin{equation}
\partial_t M(t, a) + C (t)\, \partial_ a\big(\mathcal{V}(a) \, M(t, a)\big) + \mu (a) \, M(t, a) = 0,
\end{equation}
where $C(t)$ is the concentration of modified LDL in the intima.

Monocyte transmigration from the blood flow into the intima starts with penetration, retention and modification of lipoproteins in the intima \cite{Lusis}. Moreover, stimulated by their consumption of modified LDL, macrophages secrete chemoattractants, \cite{Libby, Malat, Ross1986, Dirksen1998}. 
Consequently, we describe monocyte extravasation by the following boundary term
\begin{equation}\label{eq:bord}
M(t,0) =f  \left(C(t) ,\int_{0}^{+ \infty } B(a) \, M(t,a) \, da \right).
\end{equation}
A particular attention will be paid to low and high inflammatory states:
\begin{itemize}
\item Low inflammatory state in which case monocyte recruitment depends linearly either on macrophage concentration in the intima:  
\begin{equation}\label{ex_f_2}
M(t,0)=\int_0^{+\infty} B(a)M(t,a) da\, ,
\end{equation}
or on modified LDL concentration
\begin{equation}\label{ex_f_3}
M(t,0)=\alpha C(t)\, .
\end{equation} 
\item  Self-reinforced inflammatory state in which monocyte incoming depends on both modified LDL and macrophage concentrations in the intima:
\begin{equation}\label{ex_f_1}
M(t,0) =\sigma _m \, C(t) \left( 1+\int_{0}^{+ \infty }\mathcal{V}(a)\, M(t,a) \, da \right),
\end{equation}
where $\sigma_m$ denotes the concentration of monocytes at the endothelium.  In this case we assume that monocyte recruitment saturates as modified LDL concentration increases and chemoattractant production by macrophages is proportional to macrophage consumption of modified LDL, i.e. $B=\mathcal{V}$, which reads 

\end{itemize}

Let us turn on modified LDL dynamics in the intima that we describe by an ordinary differential equation 
\begin{equation}\label{eq:Cprime}
C'(t) = R(C(t)) - C(t)\, \int_{0}^{+ \infty } \mathcal{V}(a) \, M(t,a) \, da.
\end{equation}
 The last term in the right-hand side corresponds to lipid consumption by macrophage cells. The term $R(C(t))$ corresponds to LDL influx, modification and retention in the intima. It is known that LDL influx is closely linked to the local blood flow dynamics, see \cite{CEMRACS2}. In this work, we will focus on the the following two cases for the LDL flux: 
\begin{itemize}
\item Low endothelial injury corresponding to region with normal wall shear stress and low LDL concentration in the intima: 
%endothelial injury is due to LDL influx and LDL modification in the intima enhances LDL influx:
\begin{equation}\label{ex_R_1}
R(x) = \beta x ,
\end{equation} 
\item  High endothelial injury corresponding to region with recirculation or region with normal wall shear stress and high LDL concentration in the intima: 
%endothelial injury depends on wall shear stress and LDL influx is given by the Kedem-Katchalsky equations \cite{CEMRACS2} which can be modelled as follows:
\begin{equation}\label{ex_R}
R(x) = \gamma - \beta x.
\end{equation}   
\end{itemize}

Finally the complete coupled model is:
\begin{equation}\label{Pbcouple}
\left\{
\begin{array}{l}
\partial_t M(t, a) +  C (t) \, \partial_a (\mathcal{V}(a) \, M(t, a)) + \mu (a) \, M(t, a) = 0\, , \quad t \ge 0\, ,a> 0\, , \\
M(t,0) = f \left(C(t),\int_{0}^{+ \infty } B(a) \, M(t,a) \, da \right)\, , \quad t \ge 0\, ,\\
C'(t) = R(C(t)) - C (t) \, \int_{0}^{+ \infty } \mathcal{V}(a) \, M(t,a) \, da \, ,\quad t \ge 0\, , 
\end{array}
\right.
\end{equation}
with $f$ given by \eqref{ex_f_2} or \eqref{ex_f_3} or \eqref{ex_f_1} and $R$ given by \eqref{ex_R_1} or \eqref{ex_R} completed with an initial condition 
\begin{equation}\label{CI}
\left\{
\begin{array}{l}
M(0,.)=M_0 \in W^{1,1}(\mathbb{R}_+,\mathbb{R}_+), \\
C(0)=C_0>0.
\end{array}
\right.
\end{equation}

On the first hand, the link between inflammatory state of the lesion and $M$ is the following. The more the distribution of $M$ is concentrated for large values of $a$ the more inflammatory the plaque is and the more vulnerable it is. On the other hand the link between the model and clinical events will be as follows: the more vulnerable the plaque is and the larger the lipid amount the greater is the risk to undergo a clinical event will be.

\section{Existence and uniqueness of the solution to system \eqref{Pbcouple} }
\noindent
In this part we are interested in proving existence and uniqueness results of the solution to problem \eqref{Pbcouple}, under the following assumptions:  the functions $\mu, \, B, \, f $ and  $\mathcal{V}$ are positive and satisfy
\begin{equation}\label{hyp:M}
\left\{
\begin{array}{l}
\mu, B  \in L^{\infty}(\mathbb{R}_+), \,  ||B||_{\infty} >0 \textrm{ and } B \textrm{ is Lipschitz},\\
\mathcal{V},\, f\, : \mathbb{R}^2 \rightarrow \mathbb{R}_+\textrm{ are Lipschitz functions with } k_f>0 \textrm{ and } \mathcal{V}(0)>0,
\end{array}
\right.
\end{equation} 
%%%%%%%%%%%%%%%%%%%%%%%%%%%%%%%%%%%%%%%%%%%%%%%%%%%%%%%%%%%%%%%%%%%%%%%%%%%%%%%%%%%%%%%%%%%%%%%%%%%%%%%%%%%%%%%%%%
%%%%%%%%%%%%%%%%%%%%%%%%%%%%%%%%%%%%%%%%%%%%%%%%%%%%%%%%%%%%%%%%%%%%%%%%%%%%%%%%%%%%%%%%%%%%%%%%%%%%%%%%%%%%%%%%%%
%%%%%%%%%%%%%%%%%%%%%%%%%%%%%%%%%%%%%%%%%%%%%%%%%%%%%%%%%%%%%%%%%%%%%%%%%%%%%%%%%%%%%%%%%%%%%%%%%%%%%%%%%%%%%%%%%%
%\textbf{****Modifier ce qui suit c'est inutile de compliquer il faut traiter le cas affine cela suffit non ?****} 
the function $R : \mathbb{R} \rightarrow \mathbb{R}$ is locally Lipschitz and satisfies 
\begin{equation}\label{hyp:R}
\left\{
\begin{array}{l}
i) \, X'(t)=R(X(t)) \textrm{ with } X(0)=C_0 \textrm{ is globally bounded by }\\
K_{C_0,k} := \min\left( C_0, k\right)> 0 \textrm{ where } k:=\min \{x>0 \textrm{ s.t. } R(x)=0\},\\
ii) \, \exists S : \mathbb{R} \rightarrow \mathbb{R} \textrm{ locally Lipschitz such that } \forall x\in \mathbb{R}, \, R(x) \geq x \, S(x).
\end{array}
\right.
\end{equation} 
%%%%%%%%%%%%%%%%%%%%%%%%%%%%%%%%%%%%%%%%%%%%%%%%%%%%%%%%%%%%%%%%%%%%%%%%%%%%%%%%%%%%%%%%%%%%%%%%%%%%%%%%%%%%%%%%%%
%%%%%%%%%%%%%%%%%%%%%%%%%%%%%%%%%%%%%%%%%%%%%%%%%%%%%%%%%%%%%%%%%%%%%%%%%%%%%%%%%%%%%%%%%%%%%%%%%%%%%%%%%%%%%%%%%%
%%%%%%%%%%%%%%%%%%%%%%%%%%%%%%%%%%%%%%%%%%%%%%%%%%%%%%%%%%%%%%%%%%%%%%%%%%%%%%%%%%%%%%%%%%%%%%%%%%%%%%%%%%%%%%%%%%

In the case where $\mathcal V $ is bounded, the existence and uniqueness results are:

\begin{theorem}\label{th:global}
Assume that the functions $\mu, \, B,\, f, \, \mathcal{V}$ and $R$ satisfy \eqref{hyp:M}, \eqref{hyp:R} and that $\mathcal{V}\in L^\infty(\R_+)$. Then, there exists a unique solution $(M,C)$ in $C^0 \left(\mathbb{R}_+;L^1 (\mathbb{R}_+) \right) \times C^0 \left(\mathbb{R}_+ \right)$ to system \eqref{Pbcouple} with initial condition \eqref{CI}.
\end{theorem}

The proof of Theorem \ref{th:global} takes several steps:  subdivide the time interval and use a contraction principle to get existence and uniqueness on each subdivision of the time interval. Such a method is classical for the linear renewal equation (see \cite{Perthame_transport} e.g.). Here, to handle the non linearity on the drift, we prove a priori estimates to control the $W^{1,1}$ norm. 
%\textbf{****Insister sur ce qui est "nouveau" ou pas si standard dans notre pb (c'est plus non lineaire que d'habitude a cause du couplage en $C$ ... derivee en temps pour estimation) et voir si on ne peut pas racourcir la preuve****} 

%\bigskip
%\bigskip
%%%%%%%%%%%%%%%%%%%%%%%%%%%%%%%%%%%%%%%%%%%%%%%%%%%%%%%%%%%%%%%%%%%%%%%%%%%%%%%%%%%%%%%%%%%%%%%%%%%%%%%%%%%%%%%%%%%
%%%%%%%%%%%%%%%%%%%%%%%%%%%%%%%%%%%%%%%%%%%%%%%%%%%%%%%%%%%%%%%%%%%%%%%%%%%%%%%%%%%%%%%%%%%%%%%%%%%%%%%%%%%%%%%%%%%
%%%%%%%%%%%%%%%%%%%%%%%%%%%%%%%%%%%%%%%%%%%%%%%%%%%%%%%%%%%%%%%%%%%%%%%%%%%%%%%%%%%%%%%%%%%%%%%%%%%%%%%%%%%%%%%%%%%
%\textbf{****Modifier ou supprimer ce qui suit, mettre au plus une remarque****} 

If one does not assume that $\mathcal{V}$ is bounded, existence and uniqueness still hold true.
\begin{theorem}\label{th:global2}
Assume that the functions $\mu, \, B,\, f, \, \mathcal{V}$ and $R$ satisfy \eqref{hyp:M} and \eqref{hyp:R}. Then, there exists a unique solution $(M,C)$ in $C^0 \left(\mathbb{R}_+;L^1 (\mathbb{R}_+,(1+a)\, da) \right) \times C^0 \left(\mathbb{R}_+ \right)$ to system \eqref{Pbcouple} with initial condition \eqref{CI}.
\end{theorem}

The proof of  Theorem \ref{th:global2} follows the same lines as the proof of  Theorem \ref{th:global}, except that it is done in the space $L^1 (\mathbb{R}_+; (1+a)\, da)$ and that the evolution of $\int_0^{+\infty} M(t,a) da$  is used to control the first momentum $\int_0^{+\infty} a M(t,a) da$.
%%%%%%%%%%%%%%%%%%%%%%%%%%%%%%%%%%%%%%%%%%%%%%%%%%%%%%%%%%%%%%%%%%%%%%%%%%%%%%%%%%%%%%%%%%%%%%%%%%%%%%%%%%%%%%%%%%
%%%%%%%%%%%%%%%%%%%%%%%%%%%%%%%%%%%%%%%%%%%%%%%%%%%%%%%%%%%%%%%%%%%%%%%%%%%%%%%%%%%%%%%%%%%%%%%%%%%%%%%%%%%%%%%%%%
%%%%%%%%%%%%%%%%%%%%%%%%%%%%%%%%%%%%%%%%%%%%%%%%%%%%%%%%%%%%%%%%%%%%%%%%%%%%%%%%%%%%%%%%%%%%%%%%%%%%%%%%%%%%%%%%%%

\subsection{A priori bounds} 
\noindent

As it is classical for this kind of problem (see \cite{Perthame_transport} e.g.) we split the time interval of existence into sub-intervals and we prove existence and uniqueness on each sub-interval. Let $(T_n)_{n \geq 0}$ be a strictly increasing positive sequence with $T_0=0$. For all $n \in \mathbb{N}^*$ we take the Banach spaces $X_n$ and $Y_n$ to be
\begin{eqnarray*}
X_n=C^0 \left([T_{n-1},T_{n}];L^1 (\mathbb{R}_+,\mathbb{R}_+) \right), &\quad ||m||_{X_n} = \sup_{T_{n-1} \leq t \leq T_{n}} ||m(t,.)||_{1}, \\
Y_n=C^0 \left([T_{n-1},T_{n}] \right), & \quad ||c||_{Y_n} = \sup_{T_{n-1} \leq t \leq T_{n}} |c(t)|,
\end{eqnarray*}
%Let $n \in \mathbb{N}^*$, $m \in X_n$, 
and the operators $\mathcal{T}_n$ and $\mathcal{U}_n$ to be defined by $M := \mathcal{T}_n(m)$ and $C := \mathcal{U}_n(m)$ solutions to the system
\begin{eqnarray}
& &C '(t) = R(C (t)) - C (t) \, \int_{0}^{+ \infty } \mathcal{V}(a) \, m (t,a) \, da, \quad  t \in [T_{n-1},T_{n}], \label{Aprioribound1}\\
& &\partial_t M (t, a) +  C (t) \, \partial_a (\mathcal{V}(a) \, M (t, a)) + \mu (a) \, M (t, a) = 0, \quad  t \in [T_{n-1},T_{n}] \, , a \in \mathbb{R}_+^*, \label{Aprioribound2}\\
& &M (t,0) = f \left(C(t),\int_{0}^{+ \infty } B(a) \, M(t,a) \, da \right), \quad  t \in [T_{n-1},T_{n}], \label{Aprioribound3}\\
& &(M(T_{n-1},.),C(T_{n-1})) = \left(M_{T_{n-1}},C_{T_{n-1}}\right)\in W^{1,1}(\mathbb{R}_+) \times ]0,K_{C_0,k}]. \label{Aprioribound4}
\end{eqnarray}

\begin{proposition}\label{proposition:C}
Assume that the function $R$ satisfies assumption \eqref{hyp:R} and let $m \in X_n$. Then, there exists a unique solution $C=\mathcal{U}_n(m)$ of \eqref{Aprioribound1} with initial condition \eqref{Aprioribound4} and it satisfies
$\mathcal{U}_n(m)> 0 $ and $ ||\mathcal{U}_n(m)||_{Y_n} \leq K_{C_{0},k} < \infty$.
\end{proposition}

\begin{proof}
Using assumption \eqref{hyp:R}, we see that
\[C'(t) \geq C(t) \left(S(C(t)) - \int_0^{+ \infty} \mathcal{V}(a) m(t,a) da \right) \textrm{ on }  [T_{n-1},T_{n}]; \] 
but since $C_{T_{n-1}} >0$, $0$ is a strict subsolution of $C$,  hence $C'(t) \leq R(C(t))$ and therefore $C$ is bounded by $K_{C_0,k}$ on $[T_{n-1},T_{n}]$. 
\end{proof}

Existence and uniqueness of $\mathcal{T}_n(m)$
%, which shows that $\mathcal{T}_n:X_n \rightarrow X_n$ is well-defined, 
are stated in Proposition \ref{proposition:existencefac}, whose proof is similar to the one of Theorem \ref{th:global} in \cite{Perthame_transport}, and it is not repeated here.

\begin{proposition}\label{proposition:existencefac}
Assume that the functions $\mu, \, B,\, f, \, \mathcal{V}$ and $R$ satisfy \eqref{hyp:M}, \eqref{hyp:R}, that $\mathcal{V}\in L^\infty(\R_+)$ and let $m \in X_n$. Then, there exists a unique solution $M=\mathcal{T}_n(m)$ to the following system
\begin{equation}\label{eq:exisfac}
\left\{
\begin{array}{l}
\partial_t M (t, a) +  \mathcal{U}_n(m)(t) \, \partial_a (\mathcal{V}(a) \, M (t, a)) + \mu (a) \, M (t, a) = 0, \quad t \in [T_{n-1},T_{n}], \, a \in \mathbb{R}_+^*, \\
M (t,0) = f \left(\mathcal{U}_n(m)(t),\int_{0}^{+ \infty } B(a) \, M(t,a) \, da \right), \quad  t \in [T_{n-1},T_{n}], \\
M(T_{n-1},.) = M_{T_{n-1}}\in W^{1,1}(\mathbb{R}_+,\mathbb{R}_+).
\end{array}
\right.
\end{equation}
\end{proposition}
\medskip

In order to prove that $\mathcal{T}_n$ is a contraction on $X_n$, we first bound  $\mathcal{T}_n$ in Lemma \ref{proposition:rayon}. Next, since the non-linearity on the advection field brings a new problem,  in Lemma \ref{W1}, we provide estimates on $||\partial_a (\mathcal{V} M)||_{X_n}$.

Let us first give some notations:  
\begin{equation}\label{eq:deftheta}
\theta:=K_{C_0,k}  \, \mathcal{V}(0) \, k_f \, ||B||_{\infty}>0,
\end{equation} 
\begin{equation}\label{rayon}
\mathcal{R}_n := e^{\theta (T_{n}-T_{n-1})} \, \left(||M_{T_{n-1}}||_{1} + K_{C_0,k}  \, \mathcal{V}(0) \sup_{x \in [0,K_{C_0,k}]} f \left(x, 0 \right) \frac{1 - e^{-\theta (T_{n}-T_{n-1})}}{\theta} \right).
\end{equation}

\begin{lemma}\label{proposition:rayon}
Assume that the functions $\mu, \, B,\, f, \, \mathcal{V}$ and $R$ satisfy \eqref{hyp:M}, \eqref{hyp:R}, that $\mathcal{V}\in L^\infty(\R_+)$ and let $m \in X_n$. Then, the solution $M=\mathcal{T}_n(m)$ to the problem (\ref{Aprioribound2} -- \ref{Aprioribound4}) satisfies $$\mathcal{T}_n(m)\geq 0 \mbox{ and } \|\mathcal{T}_n(m)\|_{X_n} \leq \mathcal{R}_n < \infty.$$ Moreover, there exists $k_1>0$ and $k_2>0$ such that $\mathcal{T}_n(m) (t,0) \leq k_1 + k_2 \, \mathcal{R}_n$ on $[T_{n-1},T_{n}]$.
\end{lemma}

\begin{proof}
$0$ is a subsolution to (\ref{Aprioribound2} -- \ref{Aprioribound4}), hence $M$ is non-negative and
\[\frac{d}{dt} \int_{0}^{+ \infty } M(t, a) \, da + \int_{0}^{+ \infty } \mu (a) \, M(t, a) \, da \leq C(t)  \, \mathcal{V}(0) \, M(t,0).\]
Using that $f$ is $k_f$-Lipschitz together with Proposition \ref{proposition:C}, we can bound $M(t,0)$:
\[M(t,0)  \leq  \sup_{x \in [0,K_{C_0,k}]} f \left(x, 0 \right) + k_f \,  ||B||_{\infty} \, \int_{0}^{+ \infty } M(t,a) \, da  ,\]
therefore
\[\frac{d}{dt} \int_{0}^{+ \infty } M(t, a) \, da \leq K_{C_0,k}  \, \mathcal{V}(0) \, \left(\sup_{x \in [0,K_{C_0,k}]} f \left(x, 0 \right) + k_f \, ||B||_{\infty} \int_{0}^{+ \infty } M(t,a) \, da  \right),\] 
and Gronwall's lemma together with \eqref{eq:deftheta}  yields 
$$\|M(t,.)\|_{1} \leq e^{\theta (t-T_{n-1})} \, \left(||M_{T_{n-1}}||_{1} + K_{C_0,k}  \, \mathcal{V}(0) \sup_{x \in [0,K_{C_0,k}]} f \left(x, 0 \right) \frac{1 - e^{-\theta (t-T_{n-1})}}{\theta} \right).$$
Consequently, taking the supremum on $[T_{n-1},T_n]$, we get
\begin{equation*}
||\mathcal{T}_n(m)||_{X_n} \leq e^{\theta (T_{n}-T_{n-1})} \, \left(||M_{T_{n-1}}||_{1} + K_{C_0,k}  \, \mathcal{V}(0) \sup_{x \in [0,K_{C_0,k}]} f \left(x, 0 \right) \frac{1 - e^{-\theta (T_{n}-T_{n-1})}}{\theta} \right).
\end{equation*}
In addition we easily see that $M(t,0)  \leq  \sup_{x \in [0,K_{C_0,k}]} f \left(x, 0 \right) + k_f \,  ||B||_{\infty} \, ||\mathcal{T}_n(m)||_{X_n}$.
\end{proof}

\begin{lemma}\label{W1}
Assume that the functions $\mu, \, B,\, f, \, \mathcal{V}$ and $R$ satisfy \eqref{hyp:M}, \eqref{hyp:R}, and let $m\in \mathcal{B}_{|| .||_{X_n}}(0, \mathcal{R}_n)$. Then, there exists $\kappa_n \geq 0$ depending on $C_0$, $M_{T_{n-1}}$, $\mathcal{R}_n$  and on $T_{n} - T_{n-1}$ such that $\|\partial_a (\mathcal{V} \mathcal{T}_n(m))\|_{X_n} \leq \kappa_n$.
\end{lemma}

\begin{proof}
Let us define the function $M'$ by
\begin{equation}\label{eq:defMprime}
M'(t,a) := -C(t) \, \partial_a (\mathcal{V}(a) \, M(t,a)) - \mu(a) \, M(t,a)\, , \quad t\in[T_{n-1},T_{n}]\, , a\in \mathbb{R}_+^*\, ,
\end{equation}
using Proposition \ref{proposition:C},  we know that $C(t)>0$ for all $t \in [T_{n-1},T_{n}]$, hence
\begin{equation}\label{MajW1}
||\partial_a (\mathcal{V} \, M (t, .))||_{1} \leq \frac{1}{C(t)} \left(||M'(t,.)||_{1} + ||\mu \, M(t,.) ||_{1} \right).
\end{equation}
Furthermore, since $M$ is a solution to \eqref{Aprioribound2}, we get
\begin{equation}\label{eq:deriveM}
M'(t,a)= \partial_t M(t,a).
\end{equation}
Differentiating \eqref{Aprioribound2} with respect to the time
%, gives  
%$\partial_t M' (t, a) +  C (t) \, \partial_a (\mathcal{V}(a) \, M' (t, a)) +  C' (t) \, \partial_a (\mathcal{V}(a) \, M (t, a)) + \mu (a) \, M' (t, a) = 0$.
and recalling the definition \eqref{eq:defMprime} of $M'$ gives
\begin{equation}\label{eq:MW1}
\partial_t M' (t, a) +  C (t) \, \partial_a (\mathcal{V}(a) \, M' (t, a)) - \frac{C' (t)}{C(t)} M' (t, a) - \frac{\mu (a) C' (t)}{C(t)} \, M (t, a) + \mu (a) \, M' (t, a) = 0.
\end{equation}

As in \cite{Perthame_transport}, for $\delta >0$ let us define the $C^1$ function $I_{\delta} (u) :=  \frac{u^2}{2 \delta} $ if $ |u| \leq \delta$ and $I_{\delta} (u) := |u| - \frac{\delta}{2}$ if $ |u| \geq \delta$. For all $t \in [T_{n-1},T_{n}]$ and $a\in \mathbb{R}_+$, we have 
 \[I_{\delta}' (M' (t, a)) \, \partial_a (\mathcal{V}(a) \, M' (t, a)) =\partial_a (\mathcal{V}(a) \, I_{\delta} (M' (t, a))) + \,  \mathcal{V}'(a) \, ( M' (t, a) \, I_{\delta}' (M' (t, a)) -I_{\delta} (M' (t, a))).\]
Multiplying \eqref{eq:MW1} by $I_{\delta}' (M' (t, a))$, with the chain rule  one gets
\begin{eqnarray*}
\partial_t (I_{\delta} (M' (t, a)) )  +   C (t) \, \partial_a (\mathcal{V}(a) \, I_{\delta} (M' (t, a))) 
+ C (t) \, \mathcal{V}'(a) \, (M' (t, a) \, I_{\delta}' (M' (t, a)) - I_{\delta} (M' (t, a))) & \\
+ \left(\mu(a) - \frac{C' (t)}{C(t)} \right) \, I_{\delta} (M' (t, a)) 
+\left(\mu(a) - \frac{C' (t)}{C(t)} \right) (M' (t, a) \, I_{\delta}' (M' (t, a)) - I_{\delta} (M' (t, a))) &\\
= (I_{\delta} '(M' (t, a)) - 1) \, \frac{\mu (a) C' (t)}{C(t)} \, M (t, a) + \frac{\mu (a) C' (t)}{C(t)} \, M (t, a).&
\end{eqnarray*}
Observing that $-2 \leq I_{\delta} '(M' (t, a)) - 1 \leq 0$ and using \eqref{hyp:R}, it follows that
$(I_{\delta} '(M' (t, a)) - 1) \, \frac{C' (t)}{C(t)} \leq (I_{\delta} '(M' (t, a)) - 1) \left(S(C(t)) - \int_0^{+ \infty} \mathcal{V}(a) \, m(t,a) \, da \right) 
%\\& \leq & 
\leq 2 \left(|S(C(t))| + ||\mathcal{V} \, m(t,.)||_{1} \right)$. Since $0 \leq M' (t, a) \, I_{\delta}' (M' (t, a)) - I_{\delta} (M' (t, a)) \leq \frac{\delta}{2}$, we deduce that
\begin{eqnarray*}
\partial_t (I_{\delta} (M' (t, a)) ) +  C (t) \, \partial_a (\mathcal{V}(a) \, I_{\delta} (M' (t, a)))
 +  \left(\mu(a) - \frac{C' (t)}{C(t)} \right) I_{\delta} (M' (t, a))  &\leq &\\
 \left(2 |S(C(t))| + 2 ||\mathcal{V} \, m(t,.)||_{1}  + \frac{C' (t)}{C(t)}\right) \, \mu (a) M (t, a) 
+\left| \frac{C' (t)}{C(t)} - \mu(a) - C (t) \, \mathcal{V}'(a) \right| \, \frac{\delta}{2} \, .& &
\end{eqnarray*}
As $\delta \rightarrow 0$,  $I_{\delta} (M' (t, a)) \rightarrow |M' (t, a)|$, in $\mathcal{D}'(\R_+^2,\R_+)$, hence the following inequality holds true in the distribution sense
$\partial_t (|M' (t, a)| ) +  C (t) \, \partial_a (\mathcal{V}(a) \, |M' (t, a)|) + \left(\mu(a) - \frac{C' (t)}{C(t)} \right) |M' (t, a)| 
\leq \left(2 |S(C(t))| + 2 ||\mathcal{V} \, m(t,.)||_{1}  + \frac{C' (t)}{C(t)}\right) \, \mu (a) M (t, a)$,
which, after integration in $a$ yields that
\begin{eqnarray}\label{eq:Mp}
\frac{d}{dt} ||M'(t,.)||_{1} & \leq &  C (t) \, \mathcal{V}(0) \, |M' (t, 0)| +  \frac{C' (t)}{C(t)} ||M'(t,.)||_{1} - ||\mu \, M'(t,.)||_{1} \nonumber \\
& + & \left(2 |S(C(t))| + 2 ||\mathcal{V} \, m(t,.)||_{1}  + \frac{C' (t)}{C(t)}\right) \, || \mu \, M(t,.) ||_1 .
\end{eqnarray}

Let us define the function 
$F(t) := \frac{1}{C(t)} \left(||M'(t,.)||_{1} + ||\mu \, M(t,.) ||_{1} \right)$, $t \in [T_{n-1},T_{n}]$.
In view of \eqref{eq:Mp} together with the fact that $C>0$, one gets
\[
F'(t)\leq   \mathcal{V}(0) \, |M' (t, 0)| + \frac{2}{C(t)} \left(|S(C(t))| + ||\mathcal{V} \, m(t,.)||_{1}  \right) \, || \mu \, M(t,.) ||_1. \]
Although we can not bound $\frac{1}{C(t)}$ on $ [T_{n-1},T_{n}]$, we use that $\frac{|| \mu \, M(t,.) ||_1 }{C(t)} \leq F(t)$ and that $m\in \mathcal{B}_{|| .||_{X_n}}(0, \mathcal{R}_n)$ to get
\begin{equation}
F'(t)  \leq  \mathcal{V}(0) \, |M'(t,0)| +  2 \, \left( \sup_{ x \in [0,K_{C_0}]} |S(x)| + ||\mathcal{V}||_{\infty} \, \mathcal{R}_n \right) \, F(t). \label{eq:majF}
\end{equation}

Recalling next \eqref{eq:deriveM}, it follows that
$M'(t,0) = \dfrac{d}{dt} M(t,0) = \dfrac{d}{dt} \left( f \left(C(t),\int_{0}^{+ \infty } B(a) \, M(t,a) \, da \right) \right) $. Using the assumption \eqref{hyp:M} made on $f$ together with Proposition \ref{proposition:C}, equations \eqref{Aprioribound1} and \eqref{Aprioribound2}, we obtain that 
\begin{eqnarray}
|M'(t,0)| & \leq & k_f \, \left( |C'(t)| + \left| \int_0^{+\infty} B(a) \, \partial_t M(t,a)\,  da \right| \right) \nonumber \\
%& \leq & k_f \, \Big( \left|R(C(t)) - C(t) \int_0^{+\infty} \mathcal{V}(a) \, m(t,a) \, da \right| \nonumber  \\
%& + &   \left| C (t) \left(B(0) \mathcal{V} (0) M(t,0) + \int_0^{+\infty} B'(a) \mathcal{V}(a) M(t,a) da\right) - ||B \, \mu \, M(t,.)||_{1}  \right| \Big) \nonumber \\
& \leq & k_f \, \Big( \sup_{ x \in [0,K_{C_0}]} |R(x)| + K_{C_0} \, ||\mathcal{V}||_{\infty} \, ||m||_{X_n}  \label{in:M(0)}  \\
& + &  K_{C_0} \, \left( B(0) \, \mathcal{V} (0) \, M(t,0)  + ||B'||_{\infty} \, ||\mathcal{V}||_{\infty} \, ||M||_{X_n}  \right) + ||B||_{\infty} \, ||\mu||_{\infty} \, ||M||_{X_n}  \Big). \nonumber 
\end{eqnarray}
Proposition \ref{proposition:rayon} tells us that $||M||_{X_n} \leq \mathcal{R}_n$ and that $M(.,0)\leq k_1 + k_2 \, \mathcal{R}_n$ is bounded on $[T_{n-1},T_{n}]$. Since $m\in \mathcal{B}_{|| .||_{X_n}}(0, \mathcal{R}_n)$, the assumptions \eqref{hyp:M} and \eqref{hyp:R} make it possible to bound  the right-hand side of \eqref{in:M(0)} on $ [T_{n-1},T_{n}]$. Using now \eqref{eq:majF} and \eqref{in:M(0)} together with the previous considerations allow us to find $K_1, K_2, K_3$ and $K_4$ in $\mathbb{R}_+^*$ depending only on parameters and on $C_0$ such that 
$F'(t)  \leq  K_1 + K_2 \mathcal{R}_n +  ( K_3 + K_4 \mathcal{R}_n) \, F(t)$.
Definition \eqref{eq:defMprime} of $M'$ together with assumption \eqref{hyp:M}, $M_{T_{n-1}} \in W^{1,1} (\mathbb{R}_+)$ and $C_{T_{n-1}}>0$,  yield that 
$F(T_{n-1}) 
%& = & \frac{1}{C_{T_{n-1}}} \left(||M'_{T_{n-1}}||_{1} + ||\mu \, M_{T_{n-1}} ||_{1} \right) \\
\leq  ||\partial_a(\mathcal{V} \, M_{T_{n-1}})||_{1} + \frac{2}{C_{T_{n-1}}}  \, ||\mu \, M_{T_{n-1}} ||_{1}  < + \infty$.
Consequently, with Gronwall's lemma we deduce that 
\begin{equation*}
\sup_{t \in [T_{n-1},T_{n}]} F(t) \leq  F(T_{n-1}) \, e^{\alpha \left(\mathcal{R}_n\right) (T_{n}-T_{n-1})} + K \left(\mathcal{R}_n\right) \, \frac{e^{\alpha \left(\mathcal{R}_n\right) (T_{n}-T_{n-1})} - 1}{\alpha \left(\mathcal{R}_n\right)},
\end{equation*}
with $K \left(\mathcal{R}_n\right) = K_1 + K_2 \mathcal{R}_n $ and $\alpha \left(\mathcal{R}_n\right) = K_3 + K_4 \mathcal{R}_n$.
Defining $\kappa_n$ as 
\begin{equation}\label{eq:defkappa}
\kappa_n := F(T_{n-1}) \, e^{\alpha \left(\mathcal{R}_n\right) (T_{n}-T_{n-1})} + K \left(\mathcal{R}_n\right) \, \frac{e^{\alpha \left(\mathcal{R}_n\right) (T_{n}-T_{n-1})} - 1}{\alpha \left(\mathcal{R}_n\right)},
\end{equation}
and taking the supremum of \eqref{MajW1} on $[T_{n-1},T_{n}]$,  leads to $||\partial_a (\mathcal{V}  M)||_{X_n} \leq \sup_{t \in [T_{n-1},T_{n}]} F(t) \leq \kappa_n$.
\end{proof} 

\subsection{Existence proof}
\noindent

In the previous section, we have set the a priori estimates needed to establish the contraction of $\mathcal{T}_n$ on $X_n$.  We now need to set it on the interval $[T_{n-1},T_{n}]$. In this respect, we define the strictly increasing function: 

\begin{equation}\label{eq:defhdelta}
h_{\delta}(x) = \frac{e^{\delta \, x} - 1}{\delta} \quad \mbox{ for } \delta>0, \, x \in \mathbb{R}_+,
\end{equation}
and the sequence $(T_n)_{n \geq 0}$: 
\begin{equation}\label{T}
T_0 = 0 \textrm{ and } h_{\theta}\left(T_{n} - T_{n-1}\right) \, h_{k_R}\left(T_{n} - T_{n-1}\right) \, K_{C_0} \, ||\mathcal{V}||_{\infty} \, \left( \kappa_n + k_f \, K_{C_0} \, \mathcal{V} (0)  \right)= \frac{1}{2}.
\end{equation}

The following result is immediate.
\begin{lemma}
Assume that the functions $\mu, \, B,\, f, \, \mathcal{V}$ and $R$ satisfy \eqref{hyp:M}, \eqref{hyp:R}. Then the sequence $(T_n)_{n \geq 0}$ is strictly increasing.
\end{lemma}
%\begin{proof}
%Using assumptions \eqref{hyp:M} and \eqref{hyp:R}, we deduce that $\theta>0$, $k_R>0$, $k_f >0$, $K_{C_0}>0$ and $\mathcal{V} (0) >0$. Since $h_{\delta}$ is strictly increasing in $\mathbb{R}_+$, the definition of sequence $(T_n)_{n\geq 0}$ is then well-posed and strictly increasing.
%\end{proof}

%In Proposition \ref{proposition:existencefac}, the contraction holds on the whole space $X_n$. In the following result, we establish the contraction on a closed ball of $X_n$. 

In order to find a fixed point, we prove in the following result that $\mathcal{T}_n$ is a contraction mapping in $\mathcal{B}_{|| .||_{X_n}}(0, \mathcal{R}_n)$. 

\begin{proposition}\label{th:existN}
Under assumptions \eqref{hyp:M} and \eqref{hyp:R}, there exists a unique solution $(M,C)$ in $X_n \times Y_n$ to system \eqref{Pbcouple} with the initial condition $\left(M_{T_{n-1}},C_{T_{n-1}}\right) \in W^{1,1} (\mathbb{R}_+) \times ]0,K_{C_0}]$. 
\end{proposition}

\begin{proof}
 Let $m_1, m_2 \in \mathcal{B}_{|| .||_{X_n}}(0, \mathcal{R}_n)$. Let us define $\mathbf{m}=m_1-m_2$, $\mathbf{M} = \mathcal{T}_n(m_1) - \mathcal{T}_n(m_2)=M_1-M_2$, $\mathbf{C} = \mathcal{U}_n(m_1) - \mathcal{U}_n(m_2)=C_1-C_2$, then  
\[
\left\{\begin{array}{l}
\partial_t \mathbf{M}(t, a) +  \mathbf{C} (t) \, \partial_a (\mathcal{V}(a) \, M_1 (t, a)) + C_2 (t) \, \partial_a (\mathcal{V}(a) \, \mathbf{M}(t, a)) + \mu (a) \, \mathbf{M}(t, a) = 0, \\
\mathbf{M} (t,0) = f \left(C_1(t),\int_{0}^{+ \infty } B(a) \, M_1 (t,a) \, da \right) - f \left(C_2(t),\int_{0}^{+ \infty } B(a) \, M_2 (t,a) \, da \right), \\
\mathbf{C}'(t)=R(C_1(t)) - R(C_2(t)) - \mathbf{C}(t) \, \int_{0}^{+ \infty } \mathcal{V}(a) \, m_1 (t,a) \, da - C_2 (t) \, \int_{0}^{+ \infty } \mathcal{V}(a) \, \mathbf{m}(t,a) \, da, \\
\mathbf{M} (T_{n-1},.) = \mathbf{C}(T_{n-1})=0.
\end{array}\right.
\]
As in the proof of Lemma \ref{W1}, using the function $I_{\delta}$, it yields that
$\partial_t |\mathbf{M}(t, a)| + C_2 (t) \, \partial_a (\mathcal{V}(a)  |\mathbf{M}(t, a) |) + \mu (a) \,|\mathbf{M}(t, a)| \leq |\mathbf{C} (t)| \, |\partial_a (\mathcal{V}(a)  M_1 (t, a))|$,
hence after space integration:
\[ \frac{d}{dt} ||\mathbf{M} (t, .)||_{1} + ||\mu \, \mathbf{M}(t, .)||_{1} \leq | \mathbf{C} (t) | \, || \partial_a (\mathcal{V} \, M_1 (t, .))||_{1} 
+ C_2 (t) \, \mathcal{V}(0) \, \left|\mathbf{M}(t,0)\right|.\] 
Recalling assumption \eqref{hyp:M} 
%made on $f$, we can write $\left|\mathbf{M}(t,0)\right| \leq k_f \,(|\mathbf{C}(t)|+  ||B||_{\infty} \, ||\mathbf{M}(t, .)||_{1})$, hence
%$ \frac{d}{dt} ||\mathbf{M} (t, .)||_{1} \leq | \mathbf{C} (t) | \, || \partial_a (\mathcal{V} \, M_1 (t, .))||_{1} 
%+ C_2 (t) \, \mathcal{V}(0) \, k_f \,(|\mathbf{C}(t)|+  ||B||_{\infty} \, ||\mathbf{M}(t, .)||_{1}),$
and definition \eqref{eq:deftheta} of $\theta$, it follows that
$$\frac{d}{dt} ||\mathbf{M}(t, .)||_{1} - \theta \, ||\mathbf{M}(t, .)||_{1}  \leq  | \mathbf{C} (t) | \, \, \Big( || \partial_a (\mathcal{V} M_1 (t, .))||_{1} + k_f \, K_{C_0} \, \mathcal{V}(0) \Big).$$
Lemma \ref{W1} provides bounds on $|| \partial_a (\mathcal{V} M_1 (t, .))||_{1}$ 
%and we see that 
%$\frac{d}{dt} ||\mathbf{M}(t, .)||_{1} - \theta \, ||\mathbf{M}(t, .)||_{1}  \leq  | \mathbf{C} (t) | \, \, \Big( \kappa_n + k_f \, K_{C_0} \, \mathcal{V}(0) \Big)$.
and Gronwall's lemma together with $||\mathbf{M}(T_{n-1}, .)||_{1}=0$ gives that
$||\mathbf{M}(t, .)||_{1}  \leq  \Big( \kappa_n + k_f \, K_{C_0} \, \mathcal{V}(0) \Big) \, \int_{T_{n-1}}^t | \mathbf{C} (s) | \, e^{\theta (t-s)} \, ds$, hence taking the supremum on $[T_{n-1},T_{n}]$, one gets
\[||\mathcal{T}_n(m_1) - \mathcal{T}_n(m_2)||_{X_n}  \leq  \left( \kappa_n + k_f \, K_{C_0} \, \mathcal{V} (0)  \right) \, \frac{e^{\theta \, (T_{n} - T_{n-1})} - 1}{\theta} \,  ||\mathcal{U}_n(m_1) - \mathcal{U}_n(m_2)||_{Y_n}.\]

Let us now prove that $\mathcal{U}_n$ is locally Lipschitz. 
%First we know that
%$\mathbf{C}'(t) = R(C_1 (t)) - R(C_2 (t)) - \mathbf{C}(t) \, \int_{0}^{+ \infty } \mathcal{V}(a) \, m_1 (t,a) \, da - C_2 (t) \, \int_{0}^{+ \infty } \mathcal{V}(a) \, \mathbf{m}(t,a) \, da$,
Recalling the definition of  $\mathbf{C}$ and using the function $I_{\delta}$ defined in the proof of Lemma \ref{W1}, one gets 
\[\dfrac{d}{dt} |\mathbf{C} (t)| + |\mathbf{C}(t)| \, \int_{0}^{+ \infty } \mathcal{V}(a) \, m_1 (t,a) \, da \leq |R(C_1 (t)) - R(C_2 (t))| +C_2 (t) \,|| \mathcal{V} \, \mathbf{m}(t,.) ||_{1}.\] 
Proposition \ref{proposition:C} tells us that  for all $t \in [T_{n-1},T_n]$, $C_i(t) \leq K_{C_0}$, $i=1$ or $2$, hence using assumption \eqref{hyp:R}, we obtain that $\dfrac{d}{dt} |\mathbf{C} (t)| - k_R \, |\mathbf{C}(t)| \leq K_{C_0} \, ||\mathcal{V}||_{\infty} \, || \mathbf{m}(t,.) ||_{1}$,
which leads, after time integration, to
\[||\mathcal{U}_n(m_1) - \mathcal{U}_n(m_2)||_{Y_n} \leq K_{C_0} \, ||\mathcal{V}||_{\infty} \,  \frac{e^{ k_R  \, (T_{n} - T_{n-1})} - 1}{k_R}   \, || m_1 - m_2 ||_{X_n}.\]

Recalling next (\ref{T}), we deduce that
\begin{equation*}
|| \mathcal{T}_n(m_1)-\mathcal{T}_n(m_2) ||_{X_n} \leq \frac{1}{2} \, || m_1 - m_2 ||_{X_n},
\end{equation*}
hence, $\mathcal{T}_n$ is a contraction mapping on $\mathcal{B}_{|| .||_{X_n}}(0, \mathcal{R}_n)$, which proves the existence of a unique fixed point $M=\mathcal{T}_n (M) \in \mathcal{B}_{|| .||_{X_n}}(0, \mathcal{R}_n)$. Moreover, from Lemma \ref{proposition:rayon}, for all $m\in X_n$,  we know that  $\mathcal{T}_n (m) \in \mathcal{B}_{|| .||_{X_n}}(0, \mathcal{R}_n)$, hence  there exists a unique fixed point $M=\mathcal{T}_n (M) \in X_n$. Proposition \ref{proposition:C} provides existence and uniqueness of $C=\mathcal{U}_n (M) \in Y_n$.
\end{proof}

With Lemma \eqref{W1}, we have for all $n \in \mathbb{N}$, $\mathcal{T}_n (M)(T_n,.) \in W^{1,1}(\mathbb{R}_+,\mathbb{R}_+)$, the following result easily follows by induction. 

\begin{corollary}
Under assumptions \eqref{hyp:M} and \eqref{hyp:R}, there exists a unique solution $(M,C)$ in $C^0 \left(I;L^1 (\mathbb{R}_+,\mathbb{R}_+) \right) \times C^0 \left(I\right)$ to system (\ref{Pbcouple}) with (\ref{CI}), where $I=\bigcup_{n \geq 1} [T_{n-1},T_n]=[0, \lim_{n \rightarrow + \infty} T_n[$.
\end{corollary}

We now prove global existence for the solution to \eqref{Pbcouple} with \eqref{CI}, this will end the proof of Theorem \ref{th:global}.

\begin{proposition}\label{prop:global} 
Under assumptions \eqref{hyp:M} and \eqref{hyp:R}, the sequence $(T_n)_{n \geq 0}$ defined by \eqref{T} verifies
$\lim_{n \rightarrow + \infty} T_n = + \infty$.
\end{proposition}

\begin{proof}
We proceed in two steps. First, by induction, it is easy to prove this useful result.

\begin{lemma}\label{lem:recurrence}
Let $(u_n)_{n\geq 0}$ be a sequence satisfying the following inequality for $n \geq 1$, 
$u_n \leq e^{\alpha \,( T_{n} - T_{n-1})} \, u_{n-1} + K \, \frac{e^{\alpha \,( T_{n} - T_{n-1})} - 1}{\alpha}, \mbox{ with } \alpha,K >0$, 
then for $n \geq 0$
\begin{equation}\label{recurrence}
u_n \leq e^{\alpha \, T_{n}} \, u_{0} + K \, \frac{e^{\alpha \, T_{n}} - 1}{\alpha}.
\end{equation}
\end{lemma}

Proposition \ref{proposition:rayon} gives that $||M||_{X_{n-1}} \leq \mathcal{R}_{n-1}$, hence $||M(T_{n-1},.)||_1 = ||M_{T_{n-1}}||_1 \leq \mathcal{R}_{n-1}$. Recalling definition \eqref{rayon} of the radius $\mathcal{R}_{n}$, one gets
\[\mathcal{R}_n \leq e^{\theta (T_{n}-T_{n-1})} \, \mathcal{R}_{n-1} + K_{C_0} \, \mathcal{V}(0) \, \sup_{x \in [0,K_{C_0}]} f \left(x, 0 \right) \frac{e^{\theta (T_{n}-T_{n-1})} - 1}{\theta},\]
hence, using Lemma \ref{lem:recurrence} together with $\mathcal{R}_{0} = ||M_0||_1$
\[\mathcal{R}_n \leq e^{\theta T_{n}} \, ||M_0||_1 + K_{C_0} \, \mathcal{V}(0) \, \sup_{x \in [0,K_{C_0}]} f \left(x, 0 \right) \, \frac{e^{\theta T_{n}} - 1}{\theta}.\]

In a second step, we argue by contradiction assuming that $I$ is a bounded interval.

\begin{lemma}\label{kappan}
Assume that $\lim_{n \rightarrow + \infty} T_n = T_{max} < + \infty$, then there exist $\mathcal{K}_1,\mathcal{K}_2>0$ such that  $\mathcal{R}_n \leq \mathcal{K}_1$ and $ \kappa_n \leq \mathcal{K}_2$, for all $n \geq 0$.
\end{lemma}

\begin{proof}
Recalling definition \eqref{eq:defkappa} of $\kappa_n$ and Lemma \ref{W1}, one gets 
\[||\partial_a(\mathcal{V} M)||_{X_{n-1}} \leq \sup_{t \in [T_{n-2},T_{n-1}]} F(t) \leq \kappa_{n-1},\] 
hence, $F(T_{n-1}) \leq \kappa_{n-1}$ and 
$\kappa_n  \leq  \kappa_{n-1} \, e^{\alpha\left(\mathcal{R}_n\right) (T_{n}- T_{n-1})} + K \left(\mathcal{R}_n\right) \frac{e^{\alpha\left(\mathcal{R}_n\right) (T_{n} - T_{n-1})}-1}{\alpha\left(\mathcal{R}_n\right)}$.
Since $\lim_{n \rightarrow + \infty} T_n = T_{max} < + \infty$,  one has
$\alpha\left(\mathcal{R}_n\right) \leq K_5 $ and $ K\left(\mathcal{R}_n\right) \leq K_6$ for all $n \geq 0$.
The function $x \mapsto \frac{e^{x (T_{n} - T_{n-1})}-1}{x}$ is increasing on $\mathbb{R}_+^*$, therefore
$\kappa_n  \leq  \kappa_{n-1} \, e^{K_5 (T_{n}- T_{n-1})} + K_6 \frac{e^{K_5 (T_{n} - T_{n-1})}-1}{K_5}$.
The initial condition \eqref{CI} gives that $\kappa_{0}=F(0) < + \infty$, the result then follows from Lemma \ref{lem:recurrence}.
\end{proof}

If $\lim_{n \rightarrow + \infty} T_n = T_{max} < + \infty$, then $t_n=T_{n} - T_{n-1}$ is the positive term of a convergent series, hence $t_n \rightarrow 0$. For $\delta>0$,  using definition \eqref{eq:defhdelta} of $h_{\delta}$, one sees that  $h_{\delta} (t_n) = t_n + \circ (t_n) \mbox{ as } n \rightarrow + \infty$.  Consequently, recalling Lemma \ref{kappan} and \eqref{T}, one can find $\mathcal{K} >0$ s.t. 
\begin{eqnarray*}
\frac{1}{2} & = & \left( \kappa_n + k_f \, K_{C_0} \, \mathcal{V} (0)  \right) \, t_n ^2 \, K_{C_0} \, ||\mathcal{V}||_{L^{\infty}(\mathbb{R}_+)} + \circ (t_n^2) \\
& \leq & \mathcal{K} \, t_n^2 + \circ (t_n^2), \quad \mbox{ when } n \rightarrow + \infty.
\end{eqnarray*}
Since $t_n \rightarrow 0$, the last inequality is absurd.
\end{proof}

%%%%%%%%%%%%%%%%%%%%%%%%%%%%%%%%%%%%%%%%%%%%%%%%%%%%%%%%%%%%%%%%%%%%%%%%%%%%%%%%%%%%%%%%%%%%%%%%%%%%%%%%%%%%
%%%%%%%%%%%%%%%%%%%%%%%%%%%%%%%%%%%%%%%%%%%%%%%%%%%%%%%%%%%%%%%%%%%%%%%%%%%%%%%%%%%%%%%%%%%%%%%%%%%%%%%%%%%%
%%%%%%%%%%%%%%%%%%%%%%%%%%%%%%%%%%%%%%%%%%%%%%%%%%%%%%%%%%%%%%%%%%%%%%%%%%%%%%%%%%%%%%%%%%%%%%%%%%%%%%%%%%%%
%%%%%%%%%%%%%%%%%%%%%%%%%%%%%%%%%%%%%%%%%%%%%%%%%%%%%%%%%%%%%%%%%%%%%%%%%%%%%%%%%%%%%%%%%%%%%%%%%%%%%%%%%%%%
%%%%%%%%%%%%%%%%%%%%%%%%%%%%%%%%%%%%%%%%%%%%%%%%%%%%%%%%%%%%%%%%%%%%%%%%%%%%%%%%%%%%%%%%%%%%%%%%%%%%%%%%%%%%
%%%%%%%%%%%%%%%%%%%%%%%%%%%%%%%%%%%%%%%%%%%%%%%%%%%%%%%%%%%%%%%%%%%%%%%%%%%%%%%%%%%%%%%%%%%%%%%%%%%%%%%%%%%%
%%%%%%%%%%%%%%%%%%%%%%%%%%%%%%%%%%%%%%%%%%%%%%%%%%%%%%%%%%%%%%%%%%%%%%%%%%%%%%%%%%%%%%%%%%%%%%%%%%%%%%%%%%%%
%%%%%%%%%%%%%%%%%%%%%%%%%%%%%%%%%%%%%%%%%%%%%%%%%%%%%%%%%%%%%%%%%%%%%%%%%%%%%%%%%%%%%%%%%%%%%%%%%%%%%%%%%%%%

\section{Long time asymptotics for \eqref{Pbcouple} with \eqref{ex_f_2} or \eqref{ex_f_3}}
%low inflammatory lesions}
\noindent

In this part we study the long time asmptotics in the case of a plaque with no self reinforcement in monocyte recruitment. We assume that either $M(t,0)= \alpha C(t)$ or $M(t,0)=  \int B(a) M(t,a) da$. We will consider the two cases for the endothelium state: no injury,  $R(C)=\beta C$, and an injury due either to  wall-shear or to a high concentration of LDL in the intima, $R(C)= \gamma -\beta C$.

We will prove that in the normal endothelium case, we observe a periodic behaviour which can be biologically interpreted as plaque erodation, rupture and repair without any clinical consequence. This kind of behaviour is observed in real cases \cite{Virmani2004}.

%In the case of a disturbed wall-shear stress we will observe that .... \textbf{****A completer****}   

We first study the stationary state associated with \eqref{Pbcouple}. Then, in the particular cases of \eqref{ex_f_2} or \eqref{ex_f_3} we provide some numerical steady state profiles and we study the long-time asymptotics. 
%\textbf{****Detailler l'essence de la preuve****}

\subsection{Stationary state}
\noindent

%In this part we are interested in the stationary state \textbf{****Verifier quelles sont les hypotheses utilisees pour le calcul qui suit****}

We first prove the existence of a steady state associated with \eqref{Pbcouple} under assumptions \eqref{hyp:M} and \eqref{hyp:R}.

\subsubsection{Stationary problem}
\noindent

Let us first recall the stationary problem associated with system \eqref{Pbcouple}: find $(M^*,C^*) \in L^1(\mathbb{R}_+; (1+ \mathcal{V} (a)) \, da) \times \mathbb{R}_+ $ solution of 
\begin{equation}\label{steady}
\left\{\begin{array}{l}
C^* \, \partial_a (\mathcal{V}(a) \, M^*(a)) + \mu (a) \, M^*(a) = 0, \\
M^*(0) = f \left(C^*,\int_{0}^{+ \infty } B(a) \, M^* (a) \, da \right), \\
R(C^*) - C^* \, \int_{0}^{+ \infty } \mathcal{V}(a) \, M^*(a) \, da = 0.
\end{array}\right.
\end{equation}

\begin{lemma}\label{lem:zerosteady}
Under assumptions  \eqref{hyp:M} and \eqref{hyp:R}, the stationary problem \eqref{steady} admits $(0,0)$ for solution iff $R(0)=0$ and $f(0,0)=0$.
\end{lemma}
\begin{proof}
First we observe that if $C^*=0$, then $M^*=0$, hence $(0,0)$ is solution of \eqref{steady} if and only if $R(0)=0$ and $f(0,0)=0$. 
\end{proof} 
Let us assume that $C^*>0$, then the solution $M^*$ of $C^* \, \partial_a (\mathcal{V}(a) \, M^*(a)) + \mu (a) \, M^*(a) = 0$ is
%admits a solution $M^*\in  L^1(\mathbb{R}_+; (1+ \mathcal{V} (a)) \, da)$ given by 
\begin{equation}\label{eq:steadyM}
M^*(a) = M^*(0)\, \frac{\mathcal{V}(0)}{\mathcal{V}(a)} \, e^{- \frac{1}{C^*}  \int_{0}^{a} \frac{\mu(a')}{\mathcal{V}(a')} da'}\, ,
\end{equation}
hence, $M^*(a) >0$ as soon as $M^*(0)>0$.

\begin{lemma}\label{prop:uniquesteady}
%Under assumptions \ref{hyp:Mbis} - \ref{hyp:add} and 
Assuming that $M(t,0)=f(C(t))$ with $f$ strictly increasing and that $R(C(t))=\gamma - \beta \, C(t)$  with $\gamma >0$ and $ \beta >0$, there exists a unique solution $(M^*,C^*)$ of \eqref{steady}.
\end{lemma}
\begin{proof}
First, we define $]0,X[$ the domain of definition of the strictly increasing function $x \mapsto \int_{0}^{+ \infty } e^{- \frac{1}{x} \, \int_{0}^{a} \frac{\mu(a')}{\mathcal{V}(a')} da'}\, da$. We note that $X>0$ since the integrand is in $L^1$ for $x$ small enough
$$e^{- \frac{1}{x} \, \int_{0}^{a} \frac{\mu(a')}{\mathcal{V}(a')} da'}\leq e^{- \frac{\inf \mu}{x} \, \int_{0}^{a} \frac{da'}{\mathcal{V}(a')}} \leq e^{- \frac{\inf \mu}{x} \, \int_{0}^{a} \frac{da'}{||\mathcal{V}' ||_{\infty} a' + \mathcal{V}(0)}} \le \left(1 + \frac{||\mathcal{V}' ||_{\infty}}{\mathcal{V} (0)} \, a  \right)^{- \frac{\inf \mu }{||\mathcal{V}' ||_{\infty}  \, x}}\,, \quad a\geq 0.$$
For a given $C^*>0$, we define a unique $M^*$ by \eqref{eq:steadyM} and we substitute it in the last equation of system \eqref{steady}.  Using next that the function 
$$x \mapsto x \, f(x) \, \mathcal{V}(0) \, \int_{0}^{+ \infty } e^{- \frac{1}{x} \, \int_{0}^{a} \frac{\mu(a')}{\mathcal{V}(a')} da'}\, da + \beta x - \gamma $$ 
is continuously strictly increasing from $]0,X[$ onto $\left]-\gamma, +\infty\right[$, from the intermediate value theorem for a bijective function, we deduce that there exists a unique $C^*>0$ solution of system \eqref{steady}. Moreover we note that $ \int_{0}^{+ \infty } e^{- \frac{1}{C^*} \, \int_{0}^{a} \frac{\mu(a')}{\mathcal{V}(a')} da'}\, da = \frac{\gamma - \beta \, C^*}{C^* \, f(C^*) \, \mathcal{V}(0)} < \infty$.
\end{proof}

\begin{lemma}\label{prop:uniquesteadywithB}
%Under assumptions \ref{hyp:Mbis} - \ref{hyp:add} and 
Assuming that $M(t,0)=\int_0^{+\infty} B(a) M(t,a) da$ and that $R(C(t))=\gamma - \beta \, C(t)$ with $\gamma >0$ and $ \beta >0$, there exists a unique solution $(M^*,C^*)$ of \eqref{steady}.
\end{lemma}

\begin{proof}
Let us assume that $C^*$ and that  $M^*(0)>0$, then substituting \eqref{eq:steadyM} in the second equation of \eqref{steady}, we deduce that 
$$1 = \int_0^{+\infty} B(a) \frac{\mathcal{V}(0)}{\mathcal{V}(a)} \, e^{- \frac{1}{C^*}  \int_{0}^{a} \frac{\mu(a')}{\mathcal{V}(a')} da'} \, da \, .$$
Since $x \mapsto \int_0^{+\infty} B(a) \frac{\mathcal{V}(0)}{\mathcal{V}(a)} \, e^{- \frac{1}{x}  \int_{0}^{a} \frac{\mu(a')}{\mathcal{V}(a')} da'} \, da $ is strictly increasing from $\mathbb{R}_+^*$ onto $\mathbb{R}_+^*$, the existence and uniqueness of $C^*>0$ follows. We compute then $M^*(0)$ with the last equation of system \eqref{steady}
$$ M^*(0) = \frac{\gamma - \beta \, C^*}{C^* \mathcal{V}(0)\, \int_{0}^{+ \infty } e^{- \frac{1}{C^*} \, \int_{0}^{a} \frac{\mu(a')}{\mathcal{V}(a')} da'}\, da }.$$
\end{proof}

During the following long time asymptotic study, we will assume that $M(t,0)=f(C(t))$. We leave other cases, namely $M(t,0)$ depending on $\int_0^{+\infty} B(a) \, M(t,a) da$, for further works.

\subsubsection{Steady state profiles}
\noindent
%\textbf{****Faut-il laisser ce paragraphe ??? Si oui il faut faire des commentaires generaux du type pas de lesion tres grave****}

Here, we assume that $\mu$ does not depend on $a$, that $M(t,0)=\alpha C(t)$ and that $R(C(t))=\gamma - \beta \, C(t)$ with $\gamma >0$ and $ \beta >0$.
\begin{itemize}
\item[-]
If $\mathcal{V}(a)=\mathcal{V}$ is constant, then for all $a\ge 0$, we have
\begin{equation}\label{ES:Vct}
M^*(a)=M^*(0)  \, e^{-\frac{\mu}{\mathcal{V} \, C^*} a}\, .
\end{equation} 
The last equation of system \eqref{steady} shows that the speed $\mathcal{V}$ has an important impact on the quantity $C^*$. We notice this influence on $M^*(0)$ in the \ref{fig:ProfilexpV}.
\begin{center}
\begin{figure}
\includegraphics[scale=0.3]{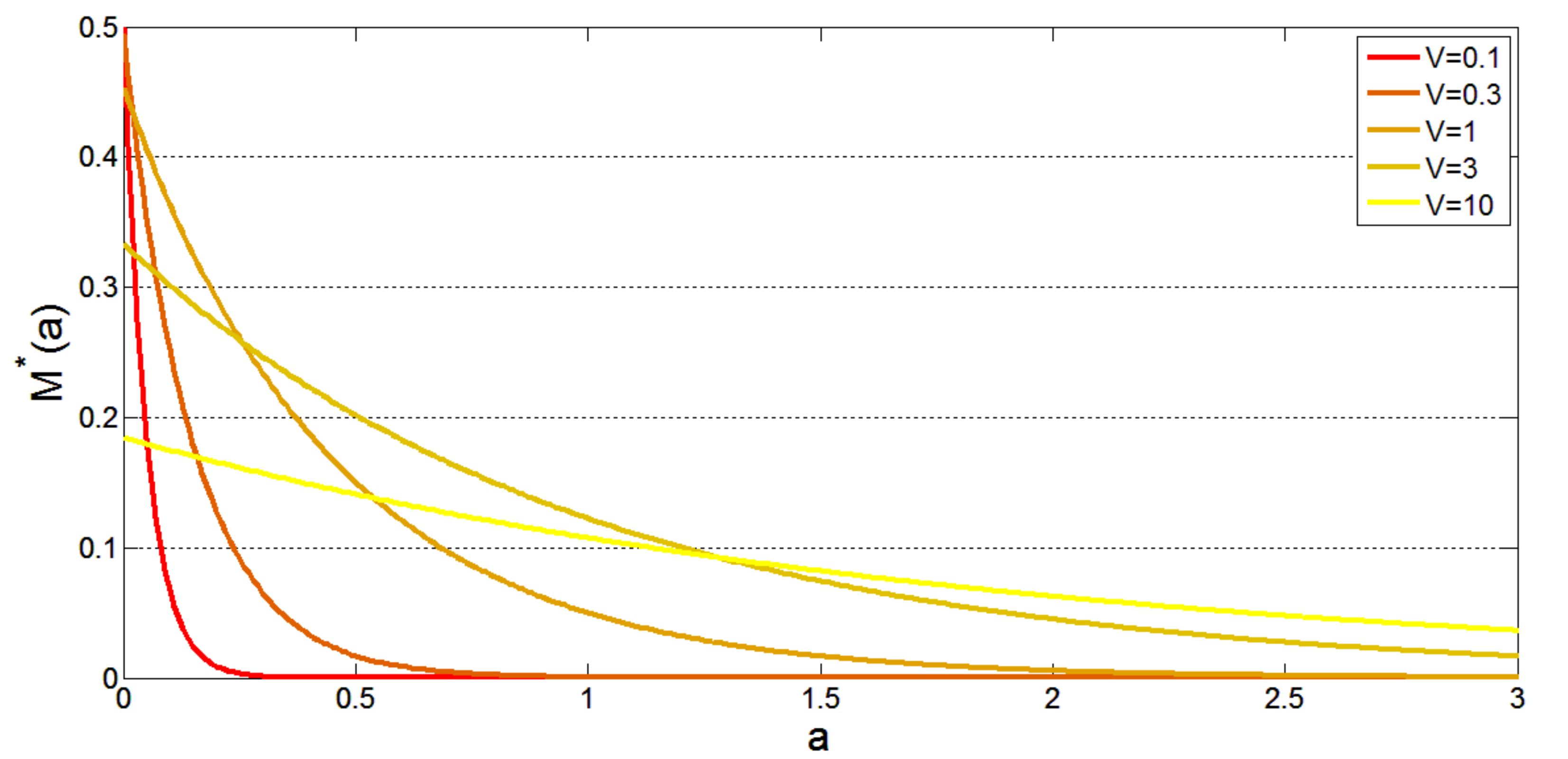}
\caption{Representation of $M^*$ solution of system \eqref{steady} for different $\mathcal{V}$ with $\alpha=\gamma=\mu=1$ and $\beta=2$.} \label{fig:ProfilexpV}
\end{figure}
\end{center}
The increase of the speed $\mathcal{V}$ leads to a repartition of macrophages more concentrated in $a$ great. The macrophages can ingest more LDL-Ox before dying. This phenomenon also leads to the decreasing of the quantity of LDL-Ox $C^*$.
\item[-]
If $\mathcal{V}(a)=\mathcal{V}_1 a + \mathcal{V}_2$ is affine,  then  for all $a\ge 0$, we have
\begin{equation}\label{ES:Vinfini}
M^*(a)=M^*(0) \, \left(1 + \frac{\mathcal{V}_1}{\mathcal{V}_2} \, a  \right)^{-\frac{ \mu }{\mathcal{V}_1  \, C^*}-1}\, .
\end{equation}
Figure \ref{fig:ProfilaffineV} shows that the decrease of $\mathcal{V}_1$ leads to a profile of $M^*$ more significant for $a$ small. Moreover for $\mathcal{V}_1 > \frac{\mu \beta}{\gamma}$, the numerical simulations show the following structure : if $\mathcal{V}_1>\mathcal{V}_1'$ then $M^*_{\mathcal{V}_1}(a) > M^*_{\mathcal{V}_1'}(a)$ for $a \geq 0$.
%%%%%%%%%%%%%%%%%%%%%%%%%%%%%%%%%%%%%%%%%%%%%%%%%%%%%%%%%%%%%%%%%%%%%%%%%%%%%%%%%%%%%%%%%%%%%%%%%%%%%%
\begin{center}
\begin{figure}
\includegraphics[scale=0.3]{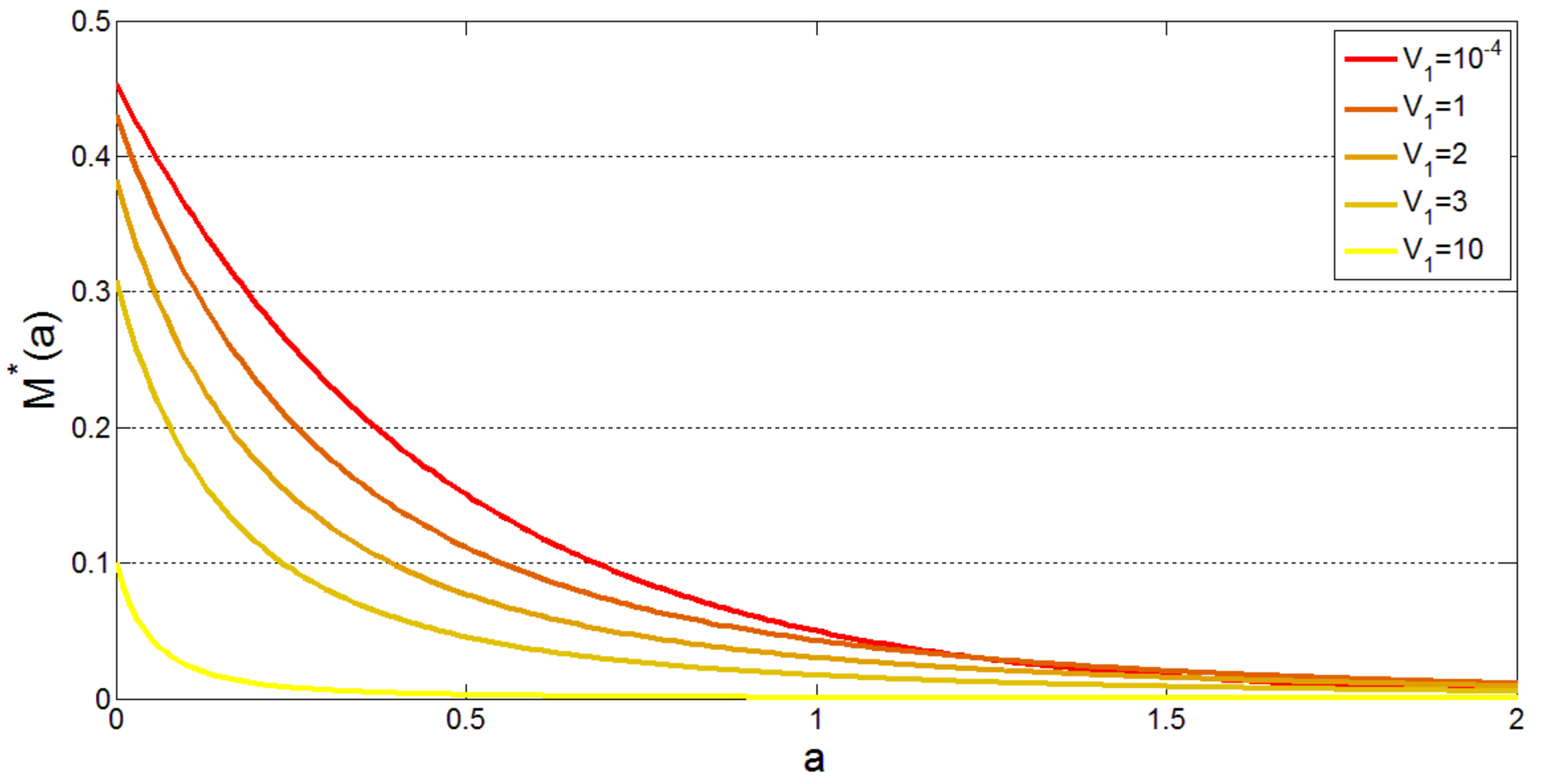}
\caption{Representation of $M^*$  solution of system \eqref{steady} for different expressions of $\mathcal{V}_1$ with $\alpha=\gamma=\mu=\mathcal{V}_2=1$ and $\beta=2$.} \label{fig:ProfilaffineV}
\end{figure}
\end{center}
%%%%%%%%%%%%%%%%%%%%%%%%%%%%%%%%%%%%%%%%%%%%%%%%%%%%%%%%%%%%%%%%%%%%%%%%%%%%%%%%%%%%%%%%%%%%%%%%%%%%%%
%hence $\frac{\mathcal{V}(0)}{\mathcal{V}(a)} \,  e^{- \frac{1}{C^*} \, \int_{0}^{a} \frac{\mu(a')}{\mathcal{V}(a')} da'}$ belongs to $ L^1(\mathbb{R}_+; (1+ \mathcal{V} (a)) \, da) $ as soon as $C^* \in \left]0,\frac{\inf_{x \geq 0} \mu (x)}{||\mathcal{V}' ||_{\infty}} \right[$.

%\begin{center}
%\begin{figure}[H]
%\includegraphics[scale=0.3]{V1V2variationPremier moment.eps} \label{fig:V1V2Variation}
%\caption{Representation of $C^*$ depending on $\mathcal{V}_1$ and $\mathcal{V}_2$ with $\alpha=\gamma=\mu=1$ and $\beta=2$. We note that $C^*$ stays below $\min \left(\frac{\mu}{\mathcal{V}_1},\frac{\gamma}{\beta}\right)$.}
%\end{figure}
%\end{center}
The first momentum of $M^*$ gives the total amount of lipid trapped inside the macrophages.
An easy computation gives $\int_0^{+\infty} a M^*(a) da=\frac{\gamma - \beta C^*}{\mu}$, so the total amount of lipid at equilibrium is
$$\int_0^{+\infty} a M^*(a) da + C^*=\frac{\gamma}{\mu} + \left(1 - \frac{\beta}{\mu}\right) C^*.$$

%%%%%%%%%%%%%%%%%%%%%%%%%%%%%%%%%%%%%%%%%%%%%%%%%%%%%%%%%%%%%%%%%%%%%%%%%%%%%%%%%%%%%%%%%%%%%%%%%%%%%%
\begin{center}
\begin{figure}
\includegraphics[scale=0.3]{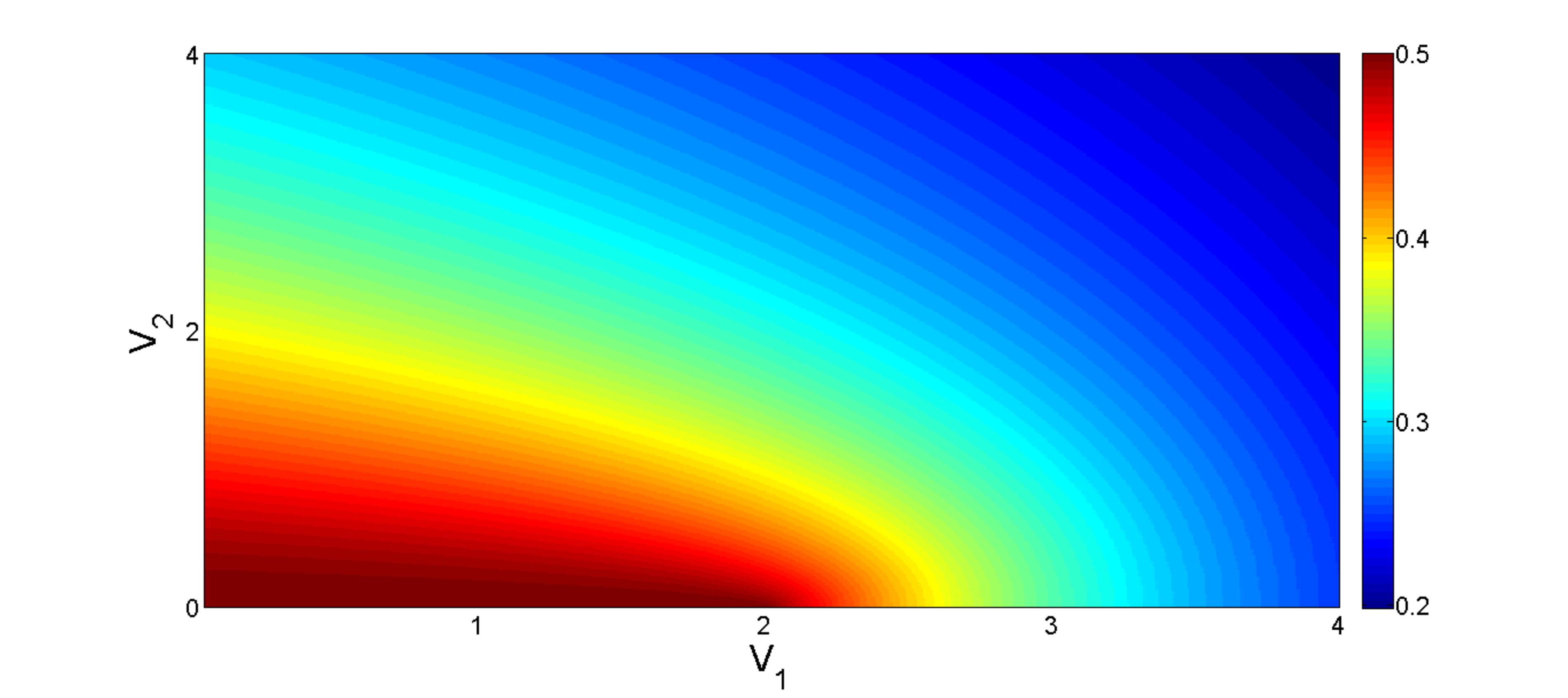} 
\caption{Representation of $C^*$ solution of system \eqref{steady} depending on $\mathcal{V}_1$ and $\mathcal{V}_2$ with $\alpha=\gamma=\mu=1$ and $\beta=2$. We note that $C^*$ stays below $\min \left(\frac{\mu}{\mathcal{V}_1},\frac{\gamma}{\beta}\right)$.} \label{fig:V1V2Variation}
\end{figure}
\end{center}
%%%%%%%%%%%%%%%%%%%%%%%%%%%%%%%%%%%%%%%%%%%%%%%%%%%%%%%%%%%%%%%%%%%%%%%%%%%%%%%%%%%%%%%%%%%%%%%%%%%%%%

The increase of the speed $\mathcal{V}$ leads to diminish the quantity of lipid component if $\beta>\mu$ and increase it if $\beta<\mu$.
Moreover, figure \ref{fig:V1V2Variation} shows that if $\mathcal{V}_2$ is small then the variation of $\mathcal{V}_1$ between $0$ and $\frac{\mu \beta}{\gamma}$ has no impact on the behaviour of the disease.
\end{itemize}

The impact of each parameter on the model is an important step to understand the full dynamics of the model.
In the rest of this section, we study the long time behaviour of the solution of system \eqref{Pbcouple} with \eqref{CI}. 

\subsection{Normal endothelium}
\noindent
%\textbf{****Pourrait-on aussi faire le cas $M(t,0)=\alpha C(t)$ ? Changer ce qui suit en mettant un theoreme****}
%\begin{remark}
 If one assumes no saturation in the monocyte recruitment, $M(t,0)=B \, \int_0^{+\infty} M(t,a) da$, a Malthusian growth model, $R(C(t))=\alpha \, C(t)$, for the LDL dynamics and that $\mu$, $\mathcal{V}$ are positive constants, then, integrating  system \eqref{Pbcouple} in $a$, denoting $\mathcal{M}(t)= \int_0^{+ \infty} \, M(t,a) \, da$, it yields the following Lotka-Volterra system
\begin{equation}\label{eq:Lotka}
\left\{
\begin{array}{l}
\mathcal{M}'(t) = B \, \mathcal{V} \, C(t) \, \mathcal{M} (t) - \mu \, \mathcal{M}(t) \\
C'(t) = \alpha \, C (t) - \mathcal{V} \, C(t) \, \mathcal{M}(t).
\end{array}
\right.
\end{equation}
The corresponding Hamiltonian function, see \cite{MurrayBio} e.g., is 
\[H(\mathcal{M}(t),C(t)) = \alpha \, \log (\mathcal{M}(t)) - \mathcal{V} \, \mathcal{M}(t) + \mu \, \log (C(t)) - B\, \mathcal{V} \, C(t).\]
We can prove that $H$ is constant along a solution $(\mathcal{M}(t),C(t))$ of system \eqref{eq:Lotka}:
\[\frac{d}{dt} \Big[ H(\mathcal{M}(t),C(t)) \Big] = 0.\]
Closed trajectories surround the steady state $\left(\dfrac{\alpha}{\mathcal{V}},\dfrac{\mu}{B \mathcal{V}}\right)$ on the level curves defined by $H$.
Since oscillations are maintained by the system, the solutions are periodic.
Such a model could be used for "young" or "low inflammatory" plaques.
%\end{remark}

%%%%%%%%%%%%%%%%%%%%%%%%%%%%%%%%%%%%%%%%%%%%%%%%%%%%%%%%%%%%%%%%%%%%%%%%%%%%%%%%%%%%%%%%%%%%%%%%%%%%%%%%%%%%
%%%%%%%%%%%%%%%%%%%%%%%%%%%%%%%%%%%%%%%%%%%%%%%%%%%%%%%%%%%%%%%%%%%%%%%%%%%%%%%%%%%%%%%%%%%%%%%%%%%%%%%%%%%%
%%%%%%%%%%%%%%%%%%%%%%%%%%%%%%%%%%%%%%%%%%%%%%%%%%%%%%%%%%%%%%%%%%%%%%%%%%%%%%%%%%%%%%%%%%%%%%%%%%%%%%%%%%%%
%%%%%%%%%%%%%%%%%%%%%%%%%%%%%%%%%%%%%%%%%%%%%%%%%%%%%%%%%%%%%%%%%%%%%%%%%%%%%%%%%%%%%%%%%%%%%%%%%%%%%%%%%%%%
%%%%%%%%%%%%%%%%%%%%%%%%%%%%%%%%%%%%%%%%%%%%%%%%%%%%%%%%%%%%%%%%%%%%%%%%%%%%%%%%%%%%%%%%%%%%%%%%%%%%%%%%%%%%
%%%%%%%%%%%%%%%%%%%%%%%%%%%%%%%%%%%%%%%%%%%%%%%%%%%%%%%%%%%%%%%%%%%%%%%%%%%%%%%%%%%%%%%%%%%%%%%%%%%%%%%%%%%%

\subsection{Injuried endothelium}
\noindent
%\textbf{****Il faut reprendre ce qui suit en supprimant ce qui est superflux et en reorganisant le paragraphe. Par exemple dans le  theoreme \ref{asymptotic_V_constant} on parle des fonctions $\mu$ et $\mathcal V$ alors qu'en fait on considere le cas ou ce sont des constantes.****}

% and then we present numerical simulations that describe the steady state profiles in some particular cases. 
Here we study the long-time asymptotics in a particular case for which  we will need the following additional assumptions to bound $M(t,0)$ globally in time and $M$ globally in $C^0(\mathbb{R}_+; L^1(\mathbb{R}_+;(1+a) \, da))$ together with $M \log(M)$ in $C^0(\mathbb{R}_+; L^1(\mathbb{R}_+;da))$: 
\begin{equation}\label{hyp:R_2}
\left\{
\begin{array}{l}
M_0 \in L^1(\mathbb{R}_+;(1+a) \, da), \, M_0 \, \log(M_0) \in L^1(\mathbb{R}_+; da),\\ \lim_{a \rightarrow +\infty} \mathcal{V}(a)^2 \, M_0(a) = 0, \, \inf_{x \geq 0} \mu(x) >0\textrm{ and }\inf_{x \geq 0} \mathcal{V}(x) >0.
\end{array}
\right.
\end{equation}

%. as it is classical \cite{Michel, MichelPerthame, Pierre}, we first looked for an entropy. The non-linear nature of the system made this study difficult and  after several unsuccessful attempts, we focused on the study of particular cases and we considered a different strategy. Starting from the observation that in some particular cases, it is possible to reduce the model to a system of ODEs, we looked for a Lyapunov function. Such a Lyapunov function provided the long time convergence of $C$ and of $\int_0^{+ \infty} \, \mathcal{V}(a) \, M(t,a) da$. We then tried to prove the convergence of $(M(t,.))_{t\geq 0}$ in $L^1$. Recalling that, in the $L^1$ space, a bounded sequence converging almost everywhere is not necessarily weakly compact in $L^1$, we needed to add some supplementary conditions to obtain the weak compactness of $(M(t,.))_{t\geq 0}$. To do so we used the Dunford-Pettis theorem. In the case where $M(t,0)=f(C(t))$, we had an explicit expression of $M$ depending only on $C$. Under the assumption that $C$ converged, the weak compactness and the convergence almost everywhere led to the $L^1$ convergence of $M$. We obtained the convergence of $C$ by reducing system \eqref{Pbcouple} in ODEs in two particular cases:  $\mathcal{V}$ constant or affine. In these two latter cases, the obtention of a Lyapunov function allowed to achieve the the proof of the convergence of $(M,C)$ solution of system \eqref{Pbcouple} with \eqref{CI}.

\begin{theorem}\label{asymptotic_V_constant}
Assume that the functions $\mu, \, B,\, f, \, \mathcal{V}$ and $R$ satisfy \eqref{hyp:M} and \eqref{hyp:R}. Moreover, we assume that $\mu>0$, $\mathcal{V}>0$ do not depend on $a$ together with $R(x)=\gamma - \beta x$ with $\gamma >0$, $\beta >0$ and that $M(t,0) = \alpha \, C(t)$, the solution $(M,C)$ of system \eqref{Pbcouple} with \eqref{CI} convergences towards $(M^*, C^*)$ the unique solution of \eqref{steady} as $t$ tends to infinity. More precisely the following strong convergence holds true 
\begin{equation*}
M(t,.) \to M^* \quad \mbox{ strongly in } L^1(\mathbb{R}_+,\mathbb{R}_+)\, , \textrm{ as } t\to +\infty\, .
\end{equation*}
\end{theorem}

The proof is broken in several steps. In a fist step, the method of characteristics is used to give an expression of $M$. 
Then, using Dunford-Pettis theorem, in proposition \ref{th:weakcompact}, we establish the weak compactness of  $ \left(M(t,\cdot)\right)_{t\ge 0}$ in $L^1$. Egorov theorem then provides the $L^1$ convergence under the assumption that $C$ converges. In the last step, in both cases considered, we reduce the system \eqref{Pbcouple} with \eqref{CI} to two different ODE systems. These systems are respectively studied in proposition \ref{propmain:Lyapunov}. Using these propositions make it possible to prove the convergence of the couple $(M,C)$.

\subsubsection{Explicit solution : method of characteristics}
\noindent

We use the method of characteristics to give an expression of $M$ solution of system \eqref{Pbcouple} with \eqref{CI}. To do so, we define $N(t,a)=\mathcal{V}(a) \, M(t,a)$  which is solution of
\begin{equation}\label{Advectiongen}
\left\{
\begin{array}{l}
\partial_t N(t, a) + C(t) \, \mathcal{V}(a) \partial_a N(t, a) + \mu (a) \, N(t, a) = 0, \qquad t \geq 0, \, a \geq 0\\
N(t,0) = \mathcal{V}(0) f(C(t)), \\
N(0,a)= \mathcal{V}(a) \, M_0(a).
\end{array}
\right.
\end{equation}
Since $\mathcal{V}$ is Lipschitz on $\mathbb{R}_+$ and recalling theorem \ref{th:global2}, we have $C\in C^0(\mathbb{R}_+,\mathbb{R}_+)$. The characteristic through the point $(t, a)\in\mathbb{R}_+^2$ is defined by the solution $\tau \mapsto \mathbf{a}_{(t,a)}(\tau)$ of the following ODE
\begin{equation}\label{Carac}
\left\{
\begin{array}{l}
\dfrac{d}{d\tau}\mathbf{a}_{(t,a)} (\tau)=C(\tau) \, \mathcal{V}(\mathbf{a}_{(t,a)} (\tau)), \quad \tau \in  \mathbb{R}_+, \\
\mathbf{a}_{(t,a)} (t) = a.
\end{array}
\right.
\end{equation}
For a given $\tau_0 \in \mathbb{R}_+$, with the construction of the characteristic $\mathbf{a}_{(t,a)}$, we deduce that 
\begin{equation}\label{eq:Caracadv}
\frac{\mathrm{d}}{\mathrm{d} \tau} \left(N(\tau, \mathbf{a}_{(t,a)} (\tau))) \, e^{\int_{\tau_0}^{\tau} \mu (\mathbf{a}_{(t,a)} (s)) ds }\right) = 0, \quad \tau \in  \mathbb{R}_+.
\end{equation}
The characteristic through $(0,0)$ cuts the $(t,a)$-plane in two regions.

%%%%%%%%%%%%%%%%%%%%%%%%%%%%%%%%%%%%%%%%%%%%%%%%%%%%%%%%%%%%%%%%%%%%%%%%%%%%%%%%%%%%%%%%%%%%%%%%%%%%%%
\begin{figure}
\begin{center}
\includegraphics[scale=0.4]{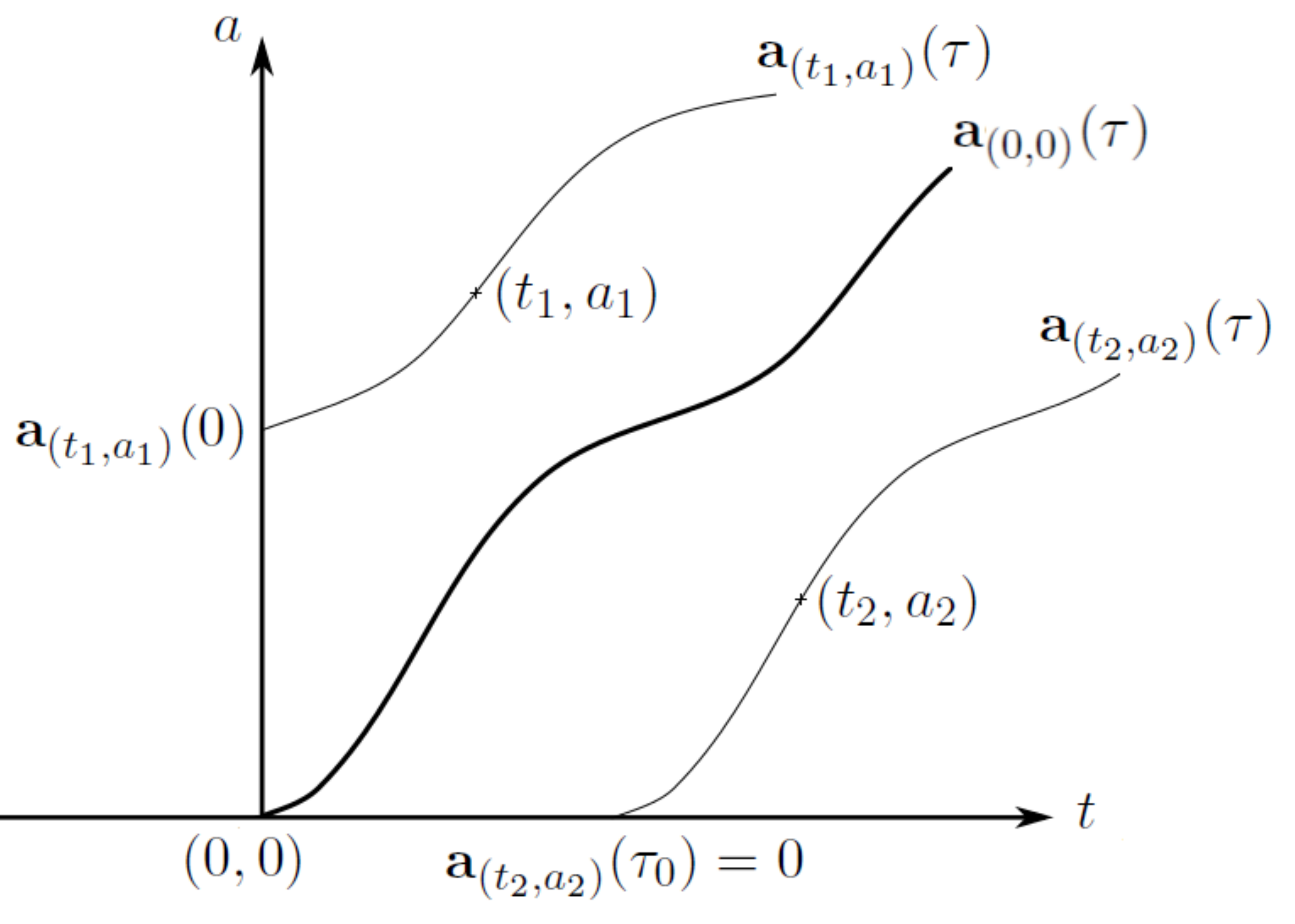}
\end{center}
\end{figure}
%%%%%%%%%%%%%%%%%%%%%%%%%%%%%%%%%%%%%%%%%%%%%%%%%%%%%%%%%%%%%%%%%%%%%%%%%%%%%%%%%%%%%%%%%%%%%%%%%%%%%%

We can define the following one-to-one mapping $\Theta$ on $\mathbb{R}_+$ by $\Theta(a) :=  \int_{0}^{a} \frac{du}{\mathcal{V}(u)}$. With proposition \ref{proposition:C}, we can also define the bijection $h(t):=\int_{0}^{t} C(s) ds$ from $\mathbb{R}_+$ to $\left[0,\int_0^{+\infty} C(s) ds\right[$.

After time integration between $t$ and $\tau_0$, the solution of \eqref{Carac} is such that
\begin{equation*}
\Theta(\mathbf{a}_{(t,a)} (\tau_0)) - \Theta(a)=\int_t^{\tau_0} C(s) \, ds = h(\tau_0)-h(t).
\end{equation*}
We see that $\mathbf{a}_{(t,a)} (\tau_0)\geq \mathbf{a}_{(0,0)} (\tau_0)$ for all $\tau_0 \in \mathbb{R}_+$ iff $\Theta(a) \geq h(t)$. Next, we use the information on the $t=0$ line by taking $\tau_0=0$ for $\Theta(a) \geq h(t)$: $\mathbf{a}_{(t,a)} (0) = \Theta^{-1}\left( \Theta(a)-h(t)\right)$,
and the information on the $a=0$ line by taking $\tau_0$ such that $\mathbf{a}_{(t,a)} (\tau_0)=0$ for $\Theta(a) \leq h(t)$: $\tau_0 = h^{-1}\left( h(t)-\Theta(a) \right)$.
Using \eqref{eq:Caracadv}, we obtain the expression of the solution of system \eqref{Advectiongen} for $(t,a)\in \mathbb{R}_+^2$
\begin{equation*}
N(t,a) = \begin{cases}
N(0,\mathbf{a}_{(t,a)} (0)) \, e^{- \int_{0}^{t} \mu \left(\mathbf{a}_{(t,a)} (s)\right)  ds} \quad\mbox{ if } \Theta(a) \geq h(t), \\
N(\tau_0,0) \, e^{- \int_{\tau_0}^{t} \mu \left(\mathbf{a}_{(t,a)} (s)\right) ds} \qquad \quad \, \mbox{ if } \Theta(a) \leq h(t). 
\end{cases}
\end{equation*}
Hence, for $(t,a) \in \mathbb{R}_+^2$, the solution $M$ of \eqref{Pbcouple} with \eqref{CI}  writes 
\begin{equation}\label{explicitM}
M(t,a) = \begin{cases}
\frac{\mathcal{V} \left(\mathbf{a}_{(t,a)} (0)\right)}{\mathcal{V}(a)} \, M_0 \left(\mathbf{a}_{(t,a)} (0)\right) \, e^{- \int_{0}^{t} \mu \left(\mathbf{a}_{(t,a)} (s)\right)  ds}, \quad\mbox{ if } \Theta(a) \geq h(t), \\
\frac{\mathcal{V}(0)}{\mathcal{V}(a)} \, f \left(C(\tau_0) \right) \, e^{- \int_{\tau_0}^{t} \mu \left(\mathbf{a}_{(t,a)} (s)\right) ds}, \quad \mbox{ if } \Theta(a) \leq h(t). 
\end{cases}
\end{equation}

\begin{lemma}\label{lem:VcarreM}
Assuming that $M(t,0) = f(C(t))$, the solution $(M,C)$ of the system \eqref{Pbcouple} with \eqref{CI} is such that $\mathcal{V}(a)^2 M(t,a) \to 0$ as $a \to +\infty$, for all $t \geq 0$.
\end{lemma}

\begin{proof}
We use the explicit expression \eqref{explicitM} of $M$ to handle this problem. For $t>0$ and for $a>\Theta^{-1}(h(t))$, we have 
\begin{eqnarray*}
\mathcal{V}(a)^2 M(t,a) & = & \mathcal{V}(a) \, \mathcal{V} \left(\mathbf{a}_{(t,a)} (0)\right)\, M_0 \left(\mathbf{a}_{(t,a)} (0)\right) \, e^{- \int_{0}^{t} \mu \left(\mathbf{a}_{(t,a)} (s)\right)  ds} \\
& \leq & \frac{\mathcal{V}(a)}{\mathcal{V} \left(\mathbf{a}_{(t,a)} (0)\right)} \, \mathcal{V} \left(\mathbf{a}_{(t,a)} (0)\right)^2 \, M_0 \left(\mathbf{a}_{(t,a)} (0)\right)
\end{eqnarray*}
%%%%%%%%%%%%%%%%%%%%%%%%%%%%%%%%%%%%%%%%%%%%%%%%%%%%%%%%%%%%%%%%%%%%%%%%%%%%%%%%%%%%%%%%%%%%%%%%%%%%%%%%%%%%
%%%%%% A CONSERVER %%%%
%%%%%%%%%%%%%%%%%%%%%%%%%%%%%%%%%%%%%%%%%%%%%%%%%%%%%%%%%%%%%%%%%%%%%%%%%%%%%%%%%%%%%%%%%%%%%%%%%%%%%%%%%%%%
%Since we take $X=a \, e^{-||\mathcal{V}'||_\infty h(t)} - \frac{\mathcal{V}(0)}{||\mathcal{V}'||_\infty} \left(1-e^{-||\mathcal{V}'||_\infty h(t)}\right)$
%\begin{eqnarray*}
%\Theta(a)-h(t) - \Theta \left(X \right) &  = & \int_X^a \frac{du}{\mathcal{V}(u)} - h(t) \geq \int_X^a \frac{du}{||\mathcal{V}'||_\infty \, u + \mathcal{V}(0)} - h(t) \\
%& \geq & \frac{1}{||\mathcal{V}'||_\infty} \log \left( \frac{||\mathcal{V}'||_\infty \, a + \mathcal{V}(0)}{||\mathcal{V}'||_\infty \, X + \mathcal{V}(0)} \right) - h(t) \geq 0.
%\end{eqnarray*}
%We have then for $a$ great enough
%\begin{equation*}
%a \, e^{-||\mathcal{V}'||_\infty h(t)} - \frac{\mathcal{V}(0)}{||\mathcal{V}'||_\infty} \left(1-e^{-||\mathcal{V}'||_\infty h(t)}\right) \leq \mathbf{a}_{(t,a)} (0)= \Theta^{-1}\left( \Theta(a)-h(t)\right) \leq a.
%\end{equation*}
%%%%%%%%%%%%%%%%%%%%%%%%%%%%%%%%%%%%%%%%%%%%%%%%%%%%%%%%%%%%%%%%%%%%%%%%%%%%%%%%%%%%%%%%%%%%%%%%%%%%%%%%%%%%
%%%%%%%%%%%%%%%%%%%%%%%%%%%%%%%%%%%%%%%%%%%%%%%%%%%%%%%%%%%%%%%%%%%%%%%%%%%%%%%%%%%%%%%%%%%%%%%%%%%%%%%%%%%%
%%%%%%%%%%%%%%%%%%%%%%%%%%%%%%%%%%%%%%%%%%%%%%%%%%%%%%%%%%%%%%%%%%%%%%%%%%%%%%%%%%%%%%%%%%%%%%%%%%%%%%%%%%%%
A computation gives that $\mathbf{a}_{(t,a)} (0)= \Theta^{-1}\left( \Theta(a)-h(t)\right) \rightarrow + \infty$ as $a \rightarrow + \infty$. We also have
\begin{eqnarray*}
 \frac{\mathcal{V}(a)}{\mathcal{V} \left(\mathbf{a}_{(t,a)} (0)\right)} & = & e^{\int_{\mathbf{a}_{(t,a)} (0)}^a \frac{\mathcal{V}'(u)}{\mathcal{V}(u)} du} 
  \leq  e^{||\mathcal{V}'||_\infty \, (\Theta(a) - \Theta (\mathbf{a}_{(t,a)} (0)))}
  =  e^{||\mathcal{V}'||_\infty h(t)}.
\end{eqnarray*}
%We use hypothesis \ref{hyp:add} made on $M_0$ to conclude.
%\mathbf{a}_{(t,a)} (0) \geq a \, \exp \left(-||\mathcal{V}'||_\infty h(t)\right) - \frac{\mathcal{V}(0)}{||\mathcal{V}'||_\infty} \left(1-\exp \left(-||\mathcal{V}'||_\infty h(t)\right)\right)
\end{proof}

%%%%%%%%%%%%%%%%%%%%%%%%%%%%%%%%%%%%%%%%%%%%%%%%%%%%%%%%%%%%%%%%%%%%%%%%%%%%%%%%%%%%%%%%%%%%%%%%%%%%%%%%%%%%

\paragraph{Weak compactness in $L^1$}
\noindent

%Let us start with some explanations on the way we handle the convergence in $L^1$. Using the method of characteristics, the solution $M$ of the system \eqref{Pbcouple} has an expression depending on the couple $\left(C,\int_0^{+\infty} B(a) \, M(.,a) \, da\right)$. In some particular cases, by reducing the system \eqref{Pbcouple} to an ODE system, it is possible to prove the convergence of this latter couple. The convergence almost everywhere of  $M$ then follows. 

In order to establish the $L^1$ convergence of $M$, we need supplementary conditions for weak compactness of a bounded sequence to handle the non reflexivity of $L^1$.
% It is to be noticed that the following result which does not require more particular assumptions than the ones we have considered in the existence part.

\begin{proposition}\label{th:weakcompact}
%Under assumptions \ref{hyp:Mbis}, \ref{hyp:R} and \ref{hyp:add}, let $M$ be the solution of system \eqref{Pbcouple} with \eqref{CI}, then
Assuming that $M(t,0) = f(C(t))$, the solution $(M,C)$ of system \eqref{Pbcouple} with \eqref{CI} is such that $(M(t,.))_{t\geq 0}$ is weakly compact in $L^1 (\mathbb{R}_+,\mathbb{R}_+)$.
\end{proposition}

We first recall the Dunford-Pettis theorem, see \cite{dunford1940linear,dunford1958linear,Bourbaki}.

\begin{theorem}\emph{\bf [Dunford-Pettis]}\label{th:Dunford}
Let $(f_t)_{t \geq 0}$ be a bounded sequence in $L^1 (\mathbb{R}_+,\mathbb{R}_+)$. The sequence
$(f_t)_{t \geq 0}$ is weakly compact in $L^1 (\mathbb{R}_+,\mathbb{R}_+)$ iff $(f_t)_{t \geq 0}$ is uniformly integrable, \textit{i. e.}
\begin{equation}\label{critere:concentration}
\forall \, \epsilon > 0, \exists \, \eta > 0 \mbox{ such that } \forall \,  \mathcal{A} \subset \mathbb{R}_+ \mbox{ with } mes(\mathcal{A}) \leq \eta, \mbox{ then } \sup_{t \in \mathbb{R}_+} \int_{\mathcal{A}} f_t(a) da \leq \epsilon,
\end{equation}
\begin{equation}\label{critere:fuite}
\lim_{A \rightarrow + \infty} \, \sup_{t \in \mathbb{R}_+} \int_{A}^{+ \infty} f_t(a) da =0.
\end{equation}
\end{theorem}

The criterion \eqref{critere:concentration} avoids concentration phenomena (as in  $f_{\epsilon} (x) = \frac{1}{2 \epsilon} \chi_{(-\epsilon,\epsilon)} (x)$ for $\epsilon>0$) and the criterion \eqref{critere:fuite} prevents the solution to "escape to $\infty$".

\begin{remark}
We can check the criterion \eqref{critere:concentration} by getting
\begin{equation}\label{critere:concentration2}
\lim_{A \rightarrow + \infty} \, \sup_{t \in \mathbb{R}_+} \int_{f_t (a) >A} f_t(a) da = 0.
\end{equation}
\end{remark}

In order to have weak compactness for $(M(t,.)_{t \geq 0}$, we will use the following convenient proposition.

\begin{proposition}\label{proposition:pratique} 
Let $(f_t)_{t \geq 0}$ be a bounded sequence in $L^1 (\mathbb{R}_+,\mathbb{R}_+)$ and let $w,G:\mathbb{R}_+ \rightarrow \mathbb{R}_+$ be such that $\lim_{a \rightarrow + \infty} w(a) = + \infty$ and  $\lim_{a\rightarrow + \infty} \frac{G(a)}{a} = + \infty$. Let us assume that $\sup_{t \in \mathbb{R}_+} \int_{0}^{+ \infty } (1 + w(a)) f_t(a) + G(f_t(a)) \, da < \infty$, then the sequence
$(f_t)_{t \geq 0}$ is weakly compact in $L^1 (\mathbb{R}_+,\mathbb{R}_+)$.
\end{proposition}

\begin{proof}
The limit $\lim_{a\rightarrow + \infty} \frac{G(a)}{a} = + \infty$ gives that $\forall \epsilon > 0$, $\exists A_0 > 0$ such that $\forall A \geq A_0$, $\frac{G(A)}{A} \geq \frac{1}{\epsilon}$. For $A \geq A_0$, we use the fact that $\sup_n \int_{0}^{+ \infty } G(f_n(a)) \, da \leq K_G$ to obtain that
$\int_{f_t(a)>A} f_t(a) da \leq \sup_{t \in \mathbb{R}_+} \int_{f_t(a)>A}  \epsilon \, G(f_t(a)) \, da  \leq  \epsilon \sup_{t \in \mathbb{R}_+} \int_{0}^{+ \infty }  G(f_t(a)) \, da  \leq \epsilon \, K_G$.
We  checking the criterion \eqref{critere:concentration2} to conclude, $\lim_{A \rightarrow + \infty } \sup_{t \in \mathbb{R}_+} \int_{f_t(a)>A} f_t(a) \, da =0$.

Similarly, the limit $\lim_{a \rightarrow + \infty} w(a) = + \infty$ provides that $\forall \epsilon > 0$, $\exists A_0 > 0$ such that $\forall A \geq A_0$, $w(A) \geq \frac{1}{\epsilon}$.
For $A \geq A_0$, we use $\sup_{t \in \mathbb{R}_+} \int_{0}^{+ \infty } w(a) \, f_t(a) \, da \leq K_w$ to get that
$\sup_{t \in \mathbb{R}_+} \int_{A}^{+ \infty} f_t(a) da \leq \epsilon \, \sup_{t \in \mathbb{R}_+} \int_{0}^{+ \infty } w(a) \, f_t(a) \, da \leq \epsilon \, K_w$.
\end{proof}

\begin{proof}[Proof of proposition \ref{th:weakcompact}]

We prove this result by using proposition \ref{proposition:pratique} with $w(a)=a$ and $G(a)=a \, | \log (a) | $ for $a \in \mathbb{R}_+$. We break the proof into several lemma.
 
\begin{lemma}\label{lem:weakL1} 
Let $(M,C)$ be the solution of system \eqref{Pbcouple} with \eqref{CI}, then we have $\sup_{t \in \mathbb{R}_+} \int_{0}^{+ \infty } M(t,a) da < \infty$ and $\sup_{t \in \mathbb{R}_+} \int_{0}^{+ \infty } a \, M(t,a) da < \infty$.
\end{lemma}

\begin{proof}
We integrate in space the equation \eqref{Pbcouple}
\begin{eqnarray*}
\frac{d}{dt} \int_{0}^{+ \infty } M(t, a) \, da +  \inf_{x \geq 0} \mu(x) \, \int_{0}^{+ \infty } M(t, a) \, da 
& \leq & C(t)  \, \mathcal{V}(0) \, f \left(C(t) \right) .
\end{eqnarray*}
We use the upper bound on $C$ to conclude with the Gronwall lemma that $\sup_{t \in \mathbb{R}_+} \int_{0}^{+ \infty } M(t,a) da < \infty$. In order to bound the first momentum, we multiply by $a \in \mathbb{R}_+$ the equation \eqref{Pbcouple}, then we integrate in space on $\mathbb{R}_+$ to get
\begin{eqnarray*}
\dfrac{d}{dt}\left( \int_{0}^{+ \infty } a \, M(t, a) \, da + C(t) \right) & + & \inf_{x \geq 0} \mu(x) \, \left(\int_{0}^{+ \infty } a \, M(t, a) \, da + C(t) \right) \\
& \leq & R(C(t)) + C(t) \, \inf_{x \geq 0} \mu(x) .
\end{eqnarray*}
Since $C$ is bounded on $\mathbb{R}_+$, we can conclude with the Gronwall lemma.
\end{proof}

We recall the following classical result. 
Its proof can be found in \cite{BDP} or \cite{Siam_CHMV} e.g.
\begin{lemma}\label{polarisation}
Let $f \in L^1 (\mathbb{R}_+;\mathbb{R}_+)$ be such that $\int_{0}^{+ \infty } (1+a) \, f(a) \, da < \infty$ and  $\int_{0}^{+ \infty } f(a) \, \log(f(a)) \, da < \infty$,
then $f \, \log(f) \in L^1 (\mathbb{R}_+,\mathbb{R})$ and 
$$\int_{0}^{+ \infty } f(a) \, |\log(f(a))| da \leq \int_{0}^{+ \infty } f(a) \, (\log(f(a)) + \delta \, a) da + \frac{4}{\delta \, e}, \quad \forall \delta > 0.$$

Moreover, the following inequality holds true 
$-\int_{0}^{+ \infty } \tilde{f}(a) \, \log(\tilde{f}(a)) da \leq \delta \, \int_{0}^{+ \infty } a \, \tilde{f}(a) da + \frac{1}{\delta \, e}$ with $\tilde{f} = f \, 1_{f \leq 1}$, $\forall \delta > 0$.
\end{lemma}

The previous result will be used to estimate $M \, | \log (M) |$.
\begin{lemma}
Let $(M,C)$ be the solution of system \eqref{Pbcouple} with \eqref{CI}, then we have \\ $\sup_{t \in \mathbb{R}_+} \int_{0}^{+ \infty } M(t,a) \, | \log (M(t,a)) | da < \infty$.
\end{lemma}
\begin{proof}

For all $t \geq 0$, we see that
\begin{equation}\label{dtNagumo}
\dfrac{d}{dt} \int_{0}^{+ \infty } M(t,a) \, \log (M(t,a)) \, da = \int_{0}^{+ \infty } \partial_t M(t,a) \, \log (M(t,a)) \, da + \int_{0}^{+ \infty } \partial_t M(t,a) \, da.
\end{equation}
Since
%& = & \left| \mathcal{V}(a) M(t,a)  \, \log(\mathcal{V}(a) M(t,a)) - \mathcal{V}(a)^2 M(t,a)  \, \frac{\log(\mathcal{V}(a))}{\mathcal{V}(a)}\right| \\
$\left|\mathcal{V}(a) M(t,a)  \, \log(M(t,a))\right| 
 \leq  \left| \mathcal{V}(a) M(t,a)  \, \log(\mathcal{V}(a) M(t,a)) \right| + \mathcal{V}(a)^2 M(t,a)  \, \sup_{X \geq \inf \mathcal{V}} \left| \frac{\log(X)}{X}\right|$,
%we use the lemma \ref{lem:weakL1} to get $\sup_{t \geq 0} \int_0^{+\infty} \mathcal{V}(a) M(t,a) < \infty$ with $\mathcal{V}(a) \leq ||\mathcal{V}'||_{\infty} \, a + \mathcal{V}(0)$. 
we use the lemma \ref{lem:VcarreM} and $\lim_{a \rightarrow +\infty} \mathcal{V}(a) M(t,a) =0$ with lemma \ref{lem:weakL1} to get \\
$\lim_{a \rightarrow +\infty} \mathcal{V}(a) M(t,a)  \log(M(t,a)) =0.$
 
Using \eqref{Pbcouple}, we compute the first term in the right-hand side of the equality \eqref{dtNagumo}:
\begin{eqnarray*}
& &\int_{0}^{+ \infty } \partial_t M(t,a) \, \log (M(t,a)) \, da  +  \int_{0}^{+ \infty } \mu(a) \,  M(t,a) \, \log (M(t,a)) \, da \\
& = & C(t)\,  M(t,0) \, \mathcal{V}(0) \, \log(M(t,0)) + C(t) \, \int_{0}^{+ \infty } \mathcal{V}(a) \, \partial_a M(t,a) \, da.
\end{eqnarray*}

Let $\tilde{M}(t,a) $ be  defined by $\tilde{M}(t,a) :=  M(t,a)$,  if $ M(t,a) \leq 1$  and $\tilde{M}(t,a) :=  0$,  if $M(t,a) > 1$  for $(t,a) \in \mathbb{R}_+^2$. We first see that 
$$-\int \mu M(t,\cdot) \log (M(t,\cdot)) 
 \le   -  \inf \mu   \int M(t,\cdot) \log (M(t,\cdot)  )
 - (  \sup \mu  -  \inf \mu )\int   \, \tilde{M}(t,\cdot) \, \log (\tilde{M}(t,\cdot)).$$
% =  - \int\mu M(t,\cdot) \log (M(t,\cdot)) _+  - \int \mu  \, \tilde{M}(t,\cdot) \, \log (\tilde{M}(t,\cdot))   
 Then, using lemma \ref{polarisation}, we obtain that 
 $$-\int \mu M(t,\cdot) \log (M(t,\cdot))  \le   -  \inf \mu   \int M(t,\cdot) \log (M(t,\cdot))   
 - (  \sup \mu  -  \inf \mu )\left( \frac{1}{e} + \int a \, M(t,\cdot )  \right).$$
 
Coming back to \eqref{dtNagumo}, after integration by parts, we obtain that
\begin{eqnarray*}
& & \frac{d}{dt} \int M(t,\cdot) \, \log (M(t,\cdot))  +   \inf \mu \, \int M(t,\cdot) \, \log (M(t,\cdot)) 
\leq  C(t) \, M(t,0) \, \mathcal{V}(0) \, \log(M(t,0)) \\
& - &  \int (C(t) \, \mathcal{V}'(\cdot) + \mu(\cdot)) M(t,\cdot)  +  \left( \sup \mu - \inf \mu \right) \, \left( \frac{1}{e} + \int a \, M(t,\cdot)  \right).
\end{eqnarray*}
%%%%%%%%%%%%%%%%%%%%%%%%%%%%%%%%%%%%%%%%%%%%%%%%%%%%%%%%%%%%%%%%%%%%%%%%%%%%%%%%%%%%%%%%%%%%%%%%%%%%%%%%%%%%
%%% CALCUL INTERMEDIAIRE %%%
%%%%%%%%%%%%%%%%%%%%%%%%%%%%%%%%%%%%%%%%%%%%%%%%%%%%%%%%%%%%%%%%%%%%%%%%%%%%%%%%%%%%%%%%%%%%%%%%%%%%%%%%%%%%
%\leq  C(t) \, M(t,0) \, \mathcal{V}(0) \, \log(M(t,0)) + \\ C(t) \, \int\mathcal{V}(\cdot) \, \partial_a M(t,\cdot)   +  \left( \sup \mu - \inf \mu \right) \, \left( \frac{1}{e} + \int a \, M(t,\cdot ) \right) -\int\left(C(t) \, \partial_a (\mathcal{V}(\cdot) M(t,\cdot)) + \mu(\cdot) \, M(t,\cdot) \right) 
%%%%%%%%%%%%%%%%%%%%%%%%%%%%%%%%%%%%%%%%%%%%%%%%%%%%%%%%%%%%%%%%%%%%%%%%%%%%%%%%%%%%%%%%%%%%%%%%%%%%%%%%%%%%
%%%%%%%%%%%%%%%%%%%%%%%%%%%%%%%%%%%%%%%%%%%%%%%%%%%%%%%%%%%%%%%%%%%%%%%%%%%%%%%%%%%%%%%%%%%%%%%%%%%%%%%%%%%%

Using lemma \ref{lem:weakL1}, we see that the last term of the latter inequality is globally bounded. Gronwall lemma, then yields that 
$$\sup_{t \in \mathbb{R}_+} \int_{0}^{+ \infty } M(t,a) \, \log (M(t,a)) da < \infty\, .$$

We achieve the proof of the proposition by using lemma  \ref{polarisation}.
\end{proof}

Consequently the hypothesis of Dunford-Pettis theorem are fulfilled, thus we conclude that $(M(t,.))_t$ is weakly compact in $L^1 (\mathbb{R}_+,\mathbb{R}_+)$. This ends the proof of proposition \ref{th:weakcompact}. 
\end{proof}

%%%%%%%%%%%%%%%%%%%%%%%%%%%%%%%%%%%%%%%%%%%%%%%%%%%%%%%%%%%%%%%%%%%%%%%%%%%%%%%%%%%%%%%%%%%%%%%%%%%%%%%%%%%%

\paragraph{Sufficient conditions for the strong convergence in $L^1$}
\noindent
 
If we assume that $C$ converges and that $M(t,0)$ depends only on $C$, then the convergence almost everywhere holds true. In such a case, the strong convergence in $L^1$ follows from the previous result, proposition  \ref{th:weakcompact}.

\begin{lemma}\label{thmain:L1convergence}
Let $(M,C)$ be the solution of the system \eqref{Pbcouple} with \eqref{CI}. Assuming in addition that $M(t,0) = f(C(t))$ and that $\lim_{t \rightarrow + \infty} C(t) = \tilde{C} $ with $\tilde{C}>0$ if $\mathcal{V}$ is constant and $\tilde{C} \in \left] 0, \frac{\inf_{x \geq 0} \mu (x)}{||\mathcal{V}' ||_{\infty}} \right[$ if $||\mathcal{V}' ||_{\infty}>0$, then  the following strong convergence holds true:
$$ M(t,a) \to  \frac{\mathcal{V}(0)}{\mathcal{V}(a)} \, f \left(\tilde{C}\right) \, e^{-\frac{1}{\tilde{C}} \, \int_{0}^{a} \frac{\mu (u)}{\mathcal{V}(u)} du}\quad \textrm{ in } L^1(\mathbb{R}_+,\mathbb{R}_+)\textrm{ when }t \rightarrow + \infty\, .$$
%Under assumptions \ref{hyp:M} or \ref{hyp:Mbis}, \ref{hyp:R} and \ref{hyp:add}, let $M$ be the solution of system \eqref{Pbcouple} with $M(t,0) = f(C(t))$ and \eqref{CI}. If we assume in addition that $\left(C(t)\right)_t$ converges and that $\lim_{t \rightarrow + \infty} C(t) = \tilde{C} $ with $\tilde{C}>0$ if $\mathcal{V}$ constant and $\tilde{C} \in \left] 0, \frac{\inf_{x \geq 0} \mu (x)}{||\mathcal{V}' ||_{\infty}} \right[$ if $||\mathcal{V}' ||_{\infty}>0$, then  the following strong convergence holds true:
%$$ M(t,a) \to  \frac{\mathcal{V}(0)}{\mathcal{V}(a)} \, f \left(\tilde{C}\right) \, e^{-\frac{1}{\tilde{C}} \, \int_{0}^{a} \frac{\mu (u)}{\mathcal{V}(u)} du}\quad \textrm{ in } L^1(\mathbb{R}_+,\mathbb{R}_+)\textrm{ when }t \rightarrow + \infty\, .$$
\end{lemma}

\begin{proof}

We first use the explicit expression \eqref{explicitM} of $M$ by substituting $v=\mathbf{a}_{(t,a)} (s)$ to obtain that $\frac{dv}{\mathcal{V}(v)}= C(s) \, ds$ with $s=h^{-1}\left(h(t)+\Theta(v)-\Theta(a)\right)$. Consequently, for $\Theta(a) \leq h(t)$, we have the expression 
\begin{equation*}
M(t,a) = \frac{\mathcal{V}(0)}{\mathcal{V}(a)} \, f \left(C(h^{-1}\left(h(t)-\Theta(a)\right) \right) \, e^{- \int_{0}^{a} \mu(v) \, \frac{1}{C \left(h^{-1}\left(h(t)+\Theta(v)-\Theta(a)\right) \right)} \frac{dv}{\mathcal{V}(v)}}.
\end{equation*}
Since $\lim_{t \rightarrow + \infty} h^{-1}\left(h(t)+\Theta(v)-\Theta(a)\right)= + \infty$ for $v \in \mathbb{R}_+$, we get the convergence almost everywhere of $M$ to a function $\tilde{M} \in L^1 (\mathbb{R}_+,\mathbb{R}_+)$ defined for $a \in \mathbb{R}_+$ by
\begin{eqnarray}\label{eq:cvaeM}
\lim_{t \rightarrow + \infty} M(t,a) & = & \frac{\mathcal{V}(0)}{\mathcal{V}(a)} \, f \left(\tilde{C}\right) \, e^{-\frac{1}{\tilde{C}} \, \int_{0}^{a} \frac{\mu (u)}{\mathcal{V}(u)} du}  = \tilde{M}(a).
\end{eqnarray}
Once we obtain the convergence almost everywhere, we achieve the proof by using the following lemma based on the Egorov's theorem, see \cite{Evans}, that we recall now.

\begin{theorem}\emph{\bf [Egorov]}\label{th:Egorov}
Let $(f_t)_{t \geq 0}$ be a sequence of functions from $\mathbb{R}_+$ to $\mathbb{R}_+$ converging almost everywhere to $f$ on a compact $[a_1,a_2] \subset \mathbb{R}_+$.  For all $\eta >0$, there exists $\mathcal{A} \subset [a_1,a_2]$ such that $mes(\mathcal{A}) < \eta$, and $(f_t)_{t \geq 0}$ converges to $f$ uniformly on $[a_1,a_2] \setminus \mathcal{A}$.
\end{theorem}

\begin{lemma}\label{lem:Egorov}
Let $(f_t)_{t \geq 0}$ be a weakly compact sequence in $L^1 (\mathbb{R}_+,\mathbb{R}_+)$ such that $(f_t)_{t \geq 0}$ converges almost everywhere to $f$ on $\mathbb{R}_+$. Then $(f_t)_{t \geq 0}$ strongly converges to $f$ in $L^1 (\mathbb{R}_+,\mathbb{R}_+)$.
\end{lemma}

\begin{proof}
For all $R > 0$, $\eta > 0$, the Egorov's theorem \ref{th:Egorov} ensures us the uniform convergence of $(f_t)_{t \geq 0}$ to $f$ on $[0,R] \setminus \mathcal{A}$ for a certain $\mathcal{A} \subset \mathbb{R}_+$ with $mes(\mathcal{A}) < \eta$. We write for $t \geq 0$
\begin{eqnarray}\label{eq:Egorovlem}
|| f_t - f ||_1 & = & \int_{[0,R] \setminus \mathcal{A}} | f_t(a) - f(a) | da + \int_{\mathcal{A}} | f_t(a) - f(a) | da + \int_{R}^{+\infty} | f_t(a) - f(a) | da \nonumber \\
& \leq & R \, || f_t - f ||_{L^{\infty} ([0,R] \setminus \mathcal{A})} + \sup_{t \in \mathbb{R}_+} \int_{\mathcal{A}} f_t(a) da + \int_{\mathcal{A}} f(a) da \nonumber \\
& + & \sup_{t \in \mathbb{R}_+} \int_{R}^{+ \infty} f_t(a) da + \int_{R}^{+ \infty} f(a) da .
\end{eqnarray}
Since $(f_t)_{t \geq 0}$ is weakly compact in $L^1 (\mathbb{R}_+,\mathbb{R}_+)$, we use the Dunford-Pettis theorem \ref{th:Dunford} to get the uniform integrability.
For $\epsilon > 0$, $\exists \, \eta >0$ such that $mes(\mathcal{A}) < \eta$ and $\frac{1}{\eta}< R < \frac{2}{\eta}$ then
$\sup_{t \in \mathbb{R}_+} \int_{\mathcal{A}} f_t(a) da \leq \frac{\epsilon}{5} $ and $  \sup_{t \in \mathbb{R}_+} \int_{R}^{+ \infty} f_t(a) da \leq \frac{\epsilon}{5}$,
and we recall that $f \in L^1 (\mathbb{R}_+,\mathbb{R}_+)$, we also have 
$\int_{\mathcal{A}} f(a) da \leq \frac{\epsilon}{5}$ and $ \int_{R}^{+ \infty} f(a) da \leq \frac{\epsilon}{5}$.

Since $\frac{1}{\eta}< R < \frac{2}{\eta}$, $\exists \, \zeta >0$ such that for all $t > \frac{1}{\zeta}$ then
$R \, || f_t - f ||_{L^{\infty} ([0,R]\setminus \mathcal{A})} \leq \frac{\epsilon}{5}$ and the $L^1$ convergence then follows from \eqref{eq:Egorovlem}.
\end{proof}

From theorem \ref{th:weakcompact} together with \eqref{eq:cvaeM} which ensures thatf lemma \ref{lem:Egorov} can be applied, we deduce that 
\begin{equation*}
M(t,a) \to  \frac{\mathcal{V}(0)}{\mathcal{V}(a)} \, f \left(\tilde{C}\right) \, e^{-\frac{1}{\tilde{C}} \, \int_{0}^{a} \frac{\mu (u)}{\mathcal{V}(u)} du} \quad \mbox{ strongly in } L^1(\mathbb{R}_+,\mathbb{R}_+)\, , \textrm{ as } t\to \infty\, .
\end{equation*}
This ends the proof of lemma \ref{thmain:L1convergence}.
\end{proof}

%When $M(t,0)=f(C(t))$ and in some particular cases of $\mathcal{V}$, we can prove the convergence of $C$, as we describe it now.

%%%%%%%%%%%%%%%%%%%%%%%%%%%%%%%%%%%%%%%%%%%%%%%%%%%%%%%%%%%%%%%%%%%%%%%%%%%%%%%%%%%%%%%%%%%%%%%%%%%%%%%%%%%%

\paragraph{Case $M(t,0)=\alpha C(t)$ and $\mathcal{V}(a) = \mathcal{V}$} %: proof of theorem \ref{asymptotic_V_constant}}
\noindent

Here we prove theorem \ref{asymptotic_V_constant}.
Denoting $\mathcal{M}(t):= \int_0^{+ \infty} \, M(t,a) \, da$,  after integration of system \eqref{Pbcouple} we get the following ODE system for $t \geq 0$
\begin{equation}\label{eq:EDOVconstant}
\left\{
\begin{array}{l}
\mathcal{M} '(t) = \alpha \, \mathcal{V} \, C(t) ^2 - \mu \, \mathcal{M}(t), \\
C'(t) = \gamma - \beta \, C (t) - \mathcal{V} \, C(t) \, \mathcal{M}(t).
\end{array}
\right.
\end{equation}
We have a Lyapunov function by handling the two non linear terms.

\begin{proposition}\label{propmain:Lyapunov}
Let $(\mathcal{M}^*,C^*)$ be the unique stationary state of \eqref{eq:EDOVconstant} in $\mathbb{R}_+^2$, then
$(\mathcal{M}^*,C^*)$ is asymptotically stable. Moreover, there exists a Lyapunov function $L_c$ in the neighbourhood of $(\mathcal{M}^*,C^*)$:
\begin{equation*}
L_c(\mathcal{M},C) := (\mathcal{M}-\mathcal{M}^*)^2 + 2 \, \alpha \, (C - C^*)^2.
\end{equation*}
\end{proposition}

\begin{proof}
According to lemma \ref{prop:uniquesteady}, there is a unique solution $(M^*,C^*) \in \mathbb{R}_+^{* 2}$ of \eqref{steady}. Let $\mathcal{M}^*$ be defined by $\mathcal{M}^*:= \int_0^\infty M^*(a)da$. We immediately see that $(\mathcal{M}^*, C^*)$ is the unique solution in $\mathbb{R}_+^{* 2}$ for the stationary problem associated to \eqref{eq:EDOVconstant}.
We prove the existence of the following Lyapunov function $L_{c}$ in the neighbourhood of $(\mathcal{M}^*,C^*)$. For $t \geq 0$, let $L_{c}(\mathcal{M}(t),C(t)) = (\mathcal{M}(t)-\mathcal{M}^*)^2 + 2 \, \alpha \, (C(t) - C^*)^2$. 
Since $R(0)>0$, we have $\mathcal{M}^*>0$ and $C^*>0$ and, for $t \geq 0$, we first compute 
\begin{eqnarray*}
& &\frac{1}{2} \, \frac{d}{dt} \left(\mathcal{M}(t) - \mathcal{M}^* \right)^2 =  \left(\alpha \, \mathcal{V} \, C(t)^2 - \mu \, \mathcal{M}(t) \right) \, \left(\mathcal{M}(t) - \mathcal{M}^* \right) \\
& = &  \left(\alpha \, \mathcal{V} \, C(t)^2 - \mu \, \mathcal{M}(t)- \left(\alpha \, \mathcal{V} \, C^{* 2} - \mu \, \mathcal{M}^*\right) \right) \, \left(\mathcal{M}(t) - \mathcal{M}^* \right) \\
& = &  \alpha \, \mathcal{V} \, \left(\left(C(t) - C^* \right)^2 + 2 \, C(t) \, C^* - 2 \, C^{* 2} \right) \, \left(\mathcal{M}(t) - \mathcal{M}^* \right)  -  \mu \, \left( \mathcal{M}(t) - \mathcal{M}^* \right)^2 \\
& = &  \alpha \, \mathcal{V} \, \left(C(t) - C^* \right)^2 \, \left(\mathcal{M}(t) - \mathcal{M}^* \right)  +  2 \, \alpha \,  \mathcal{V} \, \left(\mathcal{M}(t) - \mathcal{M}^* \right) \,  \left(C(t) - C^* \right) C^*  -  \mu \, \left( \mathcal{M}(t) - \mathcal{M}^* \right)^2. 
\end{eqnarray*}
On the other hand, we have
\begin{eqnarray*}
\frac{1}{2} \frac{d}{dt} \left(C(t) - C^* \right)^2 & = & \left(\gamma - \beta \, C(t) -  \mathcal{V}  \, C(t) \, \mathcal{M}(t) \right) \, \left(C(t) - C^* \right) \\
& = & \left(\gamma - \beta \, C(t) -  \mathcal{V} \, C(t)  \, \mathcal{M}(t) - \left(\gamma - \beta \, C^* -  \mathcal{V} \, C^* \, \mathcal{M}^*  \right) \right) \, \left(C(t) - C^* \right) \\
& = & - \beta \, (C(t) - C^*)^2 -  \mathcal{V} \, \left(C(t) \, \mathcal{M}(t) - C^* \, \mathcal{M}^* \right) \, (C(t) - C^*) \\
& = & - \beta (C(t) - C^*)^2 -  \mathcal{V}  (C(t) - C^*)^2  \mathcal{M}(t) -  \mathcal{V}  (C(t) - C^*)  C^*   (\mathcal{M}(t) - \mathcal{M}^*) .
\end{eqnarray*}
The Lyapunov function is built by cancelling the term with an unknown sign, $\mathcal{V} \,  (\mathcal{M}(t) - \mathcal{M}^*) \, (C(t) - C^*) \, C^*$, and we see that:
\begin{eqnarray*}
\frac{1}{2} \frac{d}{dt} L_{c}(\mathcal{M}(t),C(t)) 
& = &  \alpha \, \mathcal{V} \, \left(C(t) - C^*\right)^2  \, \left( \mathcal{M}(t) - \mathcal{M}^*\right) 
- \mu \, \left(\mathcal{M}(t) - \mathcal{M}^*\right)^2 \\
&  & - 2 \, \alpha \, \beta \, \left(C(t) - C^*\right)^2 
-  2 \, \alpha \, \mathcal{V} \, \left(C(t) - C^*\right)^2 \,  \mathcal{M}(t)  \\ 
& = & - \mu \, \left(\mathcal{M}(t) - \mathcal{M}^*\right)^2 - 2 \, \alpha \, \beta \, \left(C(t) - C^*\right)^2  -  \alpha \, \mathcal{V} \, \left(C(t) - C^*\right)^2 \, (\mathcal{M}(t) + \mathcal{M}^*).
\end{eqnarray*}
We use proposition \ref{proposition:rayon} to get that $\mathcal{M}(t) \geq 0$, for all $t \geq 0$. Moreover, if $(\mathcal{M}(t),C(t)) \neq (\mathcal{M}^*,C^*)$, since $\mathcal{M}^*>0$, we deduce that $\frac{d}{dt} \Big[L_{c}(\mathcal{M}(t),C(t)) \Big]<0$, hence $L_{c}$ is strictly decreasing on trajectories.  Consequently, we deduce the two convergences, $C(t)$ towards $C^*$ and $|| M(t,.) ||_1$ towards $|| M^* ||_1$. This achieves the proof of proposition \ref{propmain:Lyapunov}. 
\end{proof}

Since we know that $C$ converges to $C^*$, the proof of theorem \ref{asymptotic_V_constant} is achieved by using lemma \ref{thmain:L1convergence}.

\begin{remark}
We obtain the rate of convergence, $2 \, \min(\mu,\beta)$, of the Lyapunov function:
$$\frac{d}{dt} \Big[L_{c}(\mathcal{M}(t),C(t)) \Big]< - 2 \, \min(\mu,\beta) \, L_{c}(\mathcal{M}(t),C(t)).$$
\end{remark}

\section{Simplified high inflammatory model}
\noindent

Here we will consider the following model ($\beta= 1$):
\begin{equation}\label{Pbcouple_simplified}
\left\{
\begin{array}{l}
\partial_t M(t, a) +  C (t) \, \partial_a (\mathcal{V}(a) \, M(t, a)) + \mu (a) \, M(t, a) = 0\, , \quad t \ge 0\, ,a> 0\, , \\
M(t,0) =\sigma _m \, C(t) \left( 1+\int_{0}^{+ \infty }\mathcal{V}(a)\, M(t,a) \, da \right)\, , \quad t \ge 0\, ,\\
\end{array}
\right.
\end{equation}
where 
\begin{equation}\label{eq_C_simplified}
C(t) = \frac{\gamma }{1 + \int_{0}^{+ \infty } \mathcal{V}(a) \, M(t,a) \, da } \, ,\quad t \ge 0\, , 
\end{equation}
completed with an initial condition 
\begin{equation}\label{CI_simplified}
M(0,.)=M_0 \in W^{1,1}(\mathbb{R}_+,\mathbb{R}_+)\, .
\end{equation}
This can be rewritten as
\begin{equation}\label{Pbcouple_simplified_2}
\left\{
\begin{array}{l}
\partial_t M(t, a) +   \frac{\gamma }{1 + \int_{0}^{+ \infty } \mathcal{V}(a) \, M(t,a) \, da }  \, \partial_a (\mathcal{V}(a) \, M(t, a)) + \mu (a) \, M(t, a) = 0\, , \\
M(t,0) =\gamma \sigma _m \, .\\
\end{array}
\right.
\end{equation}

The stationary state associated with \eqref{Pbcouple_simplified_2} satisfies
\begin{equation}\label{eq:steadyM_simplified}
M^*(a) = \gamma \sigma _m \, \frac{\mathcal{V}(0)}{\mathcal{V}(a)} \, e^{- \frac{1 + \int_{0}^{+ \infty } \mathcal{V}(a) \, M^*(a) \, da}{\gamma}  \int_{0}^{a} \frac{\mu(a')}{\mathcal{V}(a')} da'}\, .
\end{equation}

Assuming that a stationary state exists, we aim at providing a bifurcation in its behavior. Tipically if $\gamma \sigma _m  <\delta_c$ the plaque is not vulnerable while it is vulnerable if $\gamma \sigma _m >\delta_c$. For us a plaque will become vulnerable if the stationary state is not decreasing in $a$. If it is decreasing for $a<A$ and then increasing for $a>A$ then this will mean there is a necrotic core and the plaque will be called vulnerable (thin wall).

Let us now compute the derivative with respect to $a$ of $M^*$ given by \eqref{eq:steadyM_simplified}:
\begin{equation}\label{eq:steadyM_simplified_derivative}
M^{*'}(a) = -\frac{\delta \, \mathcal{V}(0)}{\mathcal{V}^2(a)} \, e^{- \frac{1 + \int_{0}^{+ \infty } \mathcal{V}(a) \, M^*(a) \, da}{\gamma}  \int_{0}^{a} \frac{\mu(a')}{\mathcal{V}(a')} da'}\left( \mathcal{V}'(a)+ \frac{1 + \int_{0}^{+ \infty } \mathcal{V}(a) \, M^*(a) \, da}{\gamma}\mu (a)\right)\, .
\end{equation}
The monotony of $M^*$ depends on the sign of the last right-hand side term in the previous expression.

Assuming that $\mathcal{V}(a) = \left(1+\delta \, a \right) \, e^{-\lambda a}$ (to be justified), we have
\begin{equation}\label{eq:der_de_V}
\mathcal{V}'(a) = \left(\delta-\lambda-\lambda \, \delta \, a \right) \, e^{-\lambda a}\, ,
\end{equation}
$M^{*'}(a)$ is negative if and only if
\begin{equation}\label{eq:der_de_V_2}
 \mathcal{V}'(a)+ \frac{1 + \int_{0}^{+ \infty } \mathcal{V}(x) \, M^*(x) \, dx}{\gamma}\mu(a) \geq 0\, .
\end{equation}
We note
\begin{eqnarray*}
\Delta=\int_{0}^{+ \infty } \mathcal{V}(x) \, M^*(x) \, dx & =& \gamma \sigma _m \,  \int_{0}^{+ \infty } e^{- (1 + \Delta) \, \frac{\mu}{\gamma}  \int_{0}^{x} \frac{e^{\lambda a'}}{1+\delta \, a'} da'}\, \, dx \\
\end{eqnarray*}

\subsection{Sufficient condition for a healthy plaque ($\mu$ constant)}
\noindent
The plaque is healthy ($M^{*'}$ remains negative on $\mathbb{R}^+$) if and only if
\begin{equation}
\forall a \in \mathbb{R}^+, \quad \mathcal{V}'(a)+ \frac{1 + \Delta}{\gamma} \, \mu \geq 0\, .
\end{equation}
Since $\int_{0}^{x} \frac{e^{\lambda a'}}{1+\delta \, a'} da' \leq \frac{e^{\lambda x}-1}{\lambda}$, then
\begin{eqnarray*}
\Delta \geq  \gamma \sigma _m \,  \int_{0}^{+ \infty } e^{- (1 + \Delta) \, \frac{\mu}{\gamma}  \frac{e^{\lambda x}-1}{\lambda}}\, \, dx = \gamma \sigma _m \,  \int_{0}^{+ \infty } \frac{e^{- (1 + \Delta) \, \frac{\mu}{\gamma} u}}{1+\lambda \, u}\, \, du
\end{eqnarray*}
Since $u \mapsto \frac{1}{1+\lambda \, u}$ is convex on $\mathbb{R}_+$, we have with Jensen inequality
$$(1 + \Delta) \, \frac{\mu}{\gamma} \,  \int_{0}^{+ \infty } \frac{1}{1+\lambda \, u}\, e^{- (1 + \Delta) \, \frac{\mu}{\gamma} u} \, \, du \geq \frac{1}{1+\lambda \, (1 + \Delta) \, \frac{\mu}{\gamma} \, \int_{0}^{+ \infty }  u \, e^{- (1 + \Delta) \, \frac{\mu}{\gamma} u} \, \, du} = \frac{1}{1+\frac{\gamma \lambda}{(1 + \Delta) \, \mu}}$$ 
We can conclude $$\Delta \geq \frac{\gamma^2 \sigma_m}{(1 + \Delta) \, \mu + \lambda \gamma}.$$ 
Since $\mu \, \Delta^2 +(\mu + \lambda \gamma) \, \Delta  - \gamma^2 \sigma _m \geq 0$, then 
$$\Delta \geq \frac{-(\mu + \lambda \gamma)+\sqrt{(\mu + \lambda \gamma)^2+4 \mu \gamma^2 \sigma_m}}{2 \mu } = \frac{\mu + \lambda \gamma}{2 \mu } \left(\sqrt{1+\frac{4 \mu \gamma^2 \sigma_m}{(\mu + \lambda \gamma)^2}}-1\right).$$
Moreover, since $\sqrt{1+x}-1 \geq \frac{x}{2+\sqrt{x}}$ for $x \geq 0$, then
$$\Delta \geq \frac{\gamma^2 \sigma_m}{\mu + \lambda \gamma+\gamma\sqrt{\mu \sigma_m}}.$$
$M^{*'}$ remains negative on $\mathbb{R}^+$ if 
\begin{equation}
\forall a \in \mathbb{R}^+, \quad  \mathcal{V}'(a)+ \frac{\mu}{\gamma}+\frac{\mu \gamma \sigma_m}{\mu + \lambda \gamma+\gamma\sqrt{\mu \sigma_m}}\geq 0\,.
\end{equation}
We note
$$f(a)=\mathcal{V}'(a)+ \frac{\mu}{\gamma}+\frac{\mu \gamma \sigma_m}{\mu + \lambda \gamma+\gamma\sqrt{\mu \sigma_m}},$$ 
we have $f'(a)=\left(-2 \delta+\lambda+\lambda \, \delta \, a \right) \, \lambda \, e^{-\lambda a}$. With
$a^*=\frac{2}{\lambda}-\frac{1}{\delta}$, $f(a^*)$ is the minimum of $f$ on $\mathbb{R}^+$ if $2 \, \delta \geq \lambda$. 
So, 
\begin{itemize}
\item
if $2 \, \delta \geq \lambda$, since $e^{-2+\frac{\lambda}{\delta}} \leq 1$, we have
$$f(a^*) \geq - \delta +\frac{\mu}{\gamma}+\frac{\mu \gamma \sigma_m}{\mu + \lambda \gamma+\gamma\sqrt{\mu \sigma_m}}\geq 0....$$
%We conclude the plaque is healthy if $2 \, \delta \geq \lambda$ and $\frac{\mu}{\gamma} \geq  \delta \, e^{-2+\frac{\lambda}{\delta}}$.
\item if $2 \, \delta < \lambda$, $f$ is strictly increasing, $$f(0)=\delta-\lambda+\frac{\mu}{\gamma}+\frac{\mu \gamma \sigma_m}{\mu + \lambda \gamma+\gamma\sqrt{\mu \sigma_m}}\geq 0....$$ 
\end{itemize}
%We can sum up
%\begin{proposition}
%If ....., then the plaque is healthy.
%\end{proposition}

\subsection{Sufficient condition for a vulnerable plaque ($\mu$ constant)}
\noindent

The plaque is vulnerable ($\exists a^* \in \mathbb{R}^+, \, M^{*'}(a^*) \geq 0$) if and only if
\begin{equation}
\exists a^* \in \mathbb{R}^+, \quad \mathcal{V}'(a^*)+ \frac{1 + \Delta}{\gamma}\mu \leq 0\, .
\end{equation}
We have
\begin{eqnarray*}
\Delta=\int_{0}^{+ \infty } \mathcal{V}(x) \, M^*(x) \, dx & =& \gamma \sigma _m \,  \int_{0}^{+ \infty } e^{- (1 + \Delta) \, \frac{\mu}{\gamma}  \int_{0}^{x} \frac{e^{\lambda a'}}{1+\delta \, a'} da'}\, \, dx.
\end{eqnarray*}
Since $u \mapsto e^{\lambda u} $ is convex, we have with Jensen inequality
$$\frac{\delta}{\ln(1+\delta \, x)} \, \int_{0}^{x} \frac{e^{\lambda a'}}{1+\delta \, a'} da' \geq e^{\frac{\lambda \delta}{\ln(1+\delta \, x)} \int_{0}^{x} \frac{a'}{1+\delta \, a'} da'} = e^{\frac{\lambda \, x}{\ln(1+\delta \, x)} -\frac{\lambda}{\delta}} $$
so,
\begin{eqnarray*}
\Delta \leq  \gamma \sigma _m \,  \int_{0}^{+ \infty } e^{- (1 + \Delta) \, \frac{\mu}{\delta \gamma} \ln(1+\delta \, x) \, e^{\frac{\lambda \, x}{\ln(1+\delta \, x)} -\frac{\lambda}{\delta}}}\, \, dx 
\end{eqnarray*}
with $u=\ln(1+\delta \, x)$,
\begin{eqnarray*}
\Delta \leq  \frac{\gamma \sigma _m}{\delta} \,  \int_{0}^{+ \infty } e^{- (1 + \Delta) \, \frac{\mu}{\delta \gamma} u \, e^{\frac{\lambda}{u \delta} \left(e^u-1-u\right)}} \, e^u \, \, du
\end{eqnarray*}
Moreover, since $e^u-1 -u \geq \frac{u^2}{2}$ for $u \geq 0$, then
\begin{eqnarray*}
\Delta \leq  \frac{\gamma \sigma _m}{\delta} \,  \int_{0}^{+ \infty } e^{- (1 + \Delta) \, \frac{\mu}{\delta \gamma} u \, e^{\frac{\lambda \, u}{2 \delta}}} \, e^u \, \, du,
\end{eqnarray*}
with $t=e^{\frac{\lambda \, u}{2 \delta}}$,
\begin{eqnarray*}
\Delta \leq  \frac{2 \gamma \sigma _m}{\lambda} \,  \int_{1}^{+ \infty } e^{ \frac{2\, \ln(t)}{\lambda} \left(\delta - \frac{\lambda}{2} - (1 + \Delta) \, \frac{\mu}{\gamma} t \right)} \, \, du,
\end{eqnarray*}
%\textbf{Que faire pour majorer correctement $\Delta$ ? }

\noindent
The plaque is vulnerable if
\begin{equation}
\exists a^* \in \mathbb{R}^+, \quad \mathcal{V}'(a^*)+ \frac{\mu }{\gamma}  \leq 0\, .
\end{equation}

%\textbf{Version pr\'ec\'edente:}

With $f(a)=\mathcal{V}'(a)+ \frac{\mu}{\gamma} +  \frac{\sigma _m \,\gamma}{\min \left(1,\frac{\lambda}{\delta}\right)}$, we have $f'(a)=\left(-2 \delta+\lambda+\lambda \, \delta \, a \right) \, \lambda \, e^{-\lambda a}$. With
$a^*=\frac{2}{\lambda}-\frac{1}{\delta}$, $f(a^*)$ is the minimum of $f$ on $\mathbb{R}^+$ if $2 \, \delta \geq \lambda$. So, if $2 \, \delta \geq \lambda$
$$f(a^*)=- \delta \, e^{-\lambda a^*}+\frac{\mu}{\gamma}+  \frac{\sigma _m \,\gamma}{\min \left(1,\frac{\lambda}{\delta}\right)} \leq 0 \quad \Leftrightarrow \quad \frac{\mu}{\gamma}+ \frac{\sigma _m \,\gamma}{\min \left(1,\frac{\lambda}{\delta}\right)} \leq  \delta \, e^{-2+\frac{\lambda}{\delta}}$$
We conclude the plaque is vulnerable if $2 \, \delta \geq \lambda$ and $\frac{\mu}{\gamma}+ \frac{\sigma _m \,\gamma}{\min \left(1,\frac{\lambda}{\delta}\right)} \leq  \delta \, e^{-2+\frac{\lambda}{\delta}}$.

If $2 \, \delta < \lambda$, $f$ is strictly increasing. If $f(0)=\delta-\lambda+\frac{\mu}{\gamma}< 0$, then $\delta+\frac{\mu}{\gamma}< \lambda$. We can sum up

\begin{proposition}
If $2 \, \delta \geq \lambda$ and $\frac{\mu}{\gamma}+ \frac{\sigma _m \,\gamma}{\min \left(1,\frac{\lambda}{\delta}\right)} \leq  \delta \, e^{-2+\frac{\lambda}{\delta}}$ or $2 \, \delta < \lambda$ and $ \delta+\frac{\mu}{\gamma} < \lambda$, then the plaque is vulnerable.
\end{proposition}

\bibliographystyle{unsrt}
\bibliography{Athero}

%%%%%%%%%%%%%%%%%%%%%%%%%%%%%%%%%%%%%%%%%%%%%%%%%%%%%%%%%%%%%%%%%%%%%%%%%%%%%%%%%%%%%%%%%%%%%%%%%%%%%%%%%%%
%%%%%%%%%%%%%%%%%%%%%%%%%%%%%%%%%%%%%%%%%%%%%%%%%%%%%%%%%%%%%%%%%%%%%%%%%%%%%%%%%%%%%%%%%%%%%%%%%%%%%%%%%%%
%%%%%%%%%%%%%%%%%%%%%%%%%%%%%%%%%%%%%%%%%%%%%%%%%%%%%%%%%%%%%%%%%%%%%%%%%%%%%%%%%%%%%%%%%%%%%%%%%%%%%%%%%%%

%\end{spacing}
\end{document}